\RequirePackage{luatex85} 
\documentclass[11pt]{amsart}
\usepackage{amssymb}
\usepackage{latexsym}
\usepackage{graphicx}
\usepackage{subfigure}
\usepackage{tikz}
\usepackage[linktocpage=true]{hyperref}
\usepackage{ulem} 
\usepackage[left=3.8cm, right=3.8cm, top=3cm, bottom=3cm]{geometry}
\usepackage{color}
\usepackage{here}
\usepackage[backgroundcolor = lightgray,
            linecolor = lightgray,
            textsize = tiny]{todonotes}
\usepackage{stackengine}
\usepackage{mathrsfs}
\usepackage[all]{xy}
\usepackage{scrextend}
\usepackage{enumitem}
\usepackage{comment}

\hypersetup{ colorlinks   = true, 
urlcolor  = teal, 
linkcolor    = magenta, 
citecolor   = cyan 
}

\def\a{{\alpha}}
\def\e{{\varepsilon}}
\def\beq{\begin{equation}}
\def\eeq{\end{equation}}

\def\a{\alpha}

\def\e{\epsilon}

\newcommand{\cC}{{\mathcal C}}
\newcommand{\cD}{{\mathcal D}}

\newcommand{\cB}{{\mathcal B}}

\newcommand{\cW}{{\mathcal W}}

\newcommand{\cPH}{{\mathcal{PH}}}

\newcommand{\cU}{{\mathcal U}}

\newcommand{\cO}{{\mathcal O}}

\newcommand{\wt}{\widetilde}

\newcommand{\R}{{\mathbb R}}

\newcommand{\N}{{\mathbb N}}

\newtheorem{thm}{Theorem}[section]
\newtheorem{rem}[thm]{Remark}

\newtheorem{lem}[thm]{Lemma}
\newtheorem{conj}[thm]{Conjecture}
\newtheorem{prop}[thm]{Proposition}
\newtheorem{cor}[thm]{Corollary}
\newtheorem{claim}[thm]{Claim}

\theoremstyle{definition}
\newtheorem{df}[thm]{Definition}

\theoremstyle{plain}
\newtheorem*{main thm}{Main Theorem}
\newtheorem*{main thm bis}{Main Theorem (alternate version)}

\newtheorem{theoalph}{Theorem}

\newtheorem{question}{Question}
\providecommand{\norm}[1]{\lVert#1\rVert}

\title{Robust transitivity of geodesic flows from metrics with conjugate points}
\author{Ygor de Jesus, Luis Pedro Piñeyrúa, Sergio Romaña}

\newcommand{\Addresses}{{
  \bigskip
  \footnotesize

\noindent \textbf{Y. de Jesus}\\
 Department of Mathematics\\
University of Luxembourg\\
Esch-sur-Alzette, Luxembourg\\
\textit{E-mail}: \texttt{ygor.dejesus@uni.lu}

\noindent\textbf{L. P. Piñeyrúa}\\
IMERL, Facultad de Ingeniería\\
Universidad de la República\\
Montevideo, Uruguay\\
\textit{E-mail}: \texttt{lpineyrua@fing.edu.uy}

\noindent\textbf{S. Romaña}\\
School of  Mathematics (Zhuhai)\\
Sun Yat-sen University\\
Zhuhai, 519082, PR China\\
\textit{E-mail}: \texttt{sergio@mail.sysu.edu.cn}
}}

\begin{document}
\numberwithin{equation}{section}

\maketitle

\begin{abstract}
    In this article, we prove a general criterion for obtaining robust transitivity of partially hyperbolic geodesic flows. As a consequence, we exhibit the first example of a $C^2$-open set of Riemannian metrics with conjugate points and transitive geodesic flow. In particular, such a set does not intersect the set of Riemannian metrics with Anosov geodesic flows. 
\end{abstract}

\tableofcontents

\section{Introduction}
\subsection{Main results}
Given $(M,g)$ a closed Riemannian manifold of arbitrary dimension, we define the geodesic flow associated to it by the following $\R$-action on the unit tangent bundle $T^1M=\{v\in TM: g(v,v)=1\}$:
\begin{align*}
    \varphi^t: T^1M& \rightarrow T^1M\\
    v&\mapsto \gamma_v'(v).
\end{align*}
This dynamical system establishes deep connections between dynamical properties and the geometry of the underlying manifold. This interplay goes in both directions, where geometric properties imply dynamical properties, and vice versa. 

For example, a classical result states that if a metric $g$ has negative sectional curvature, then its geodesic flow is uniformly hyperbolic (an Anosov flow) \cite{An}. This means that the tangent bundle along orbits splits into complementary invariant directions in which the dynamics is simplified: distances are exponentially contracted in one direction and exponentially expanded in the other (see Subsection \ref{sAnosov} for precise definitions).
On the other hand, Klingenberg \cite{Kl} proved that metrics with Anosov geodesic flow on compact manifolds cannot present geodesics with conjugate points (a purely geometrical feature). This result was later extended to complete metrics on not necessarily compact manifolds by Ma\~n\'e \cite{Mane} and for general non-compact surfaces in \cite{MR} by Melo and Roma\~na. 

We remark that it is not true that metrics with Anosov geodesic flows must present only non-positive curvature; explicit examples of Riemannian metrics with curvature with both signs were constructed in \cite{DonnayPugh2003} and \cite{GU} (see also \cite{Eb1}).

Given these rich interplays between the Anosov property of geodesic flow and the geometry of the manifold, it is natural to pose the following general question:

\begin{question}
    Which geometric properties can or cannot coexist with which dynamical property of the geodesic flow?
\end{question}

More specifically:

\begin{question}
    Which ``interesting'' dynamical properties of the geodesic flow can occur outside the class of Anosov geodesic flows? 
\end{question}

This second question became a focal point in several works published in the early 2000s. For instance, every metric on the sphere $\mathbb{S}^2$ must admit conjugate points; hence, no geodesic flow on $\mathbb{S}^2$ can be Anosov (similar result for $\mathbb{T}^2$). On the other hand, significant progress was made in constructing metrics on $\mathbb{S}^2$ whose geodesic flows exhibit rich dynamics---such as positive topological entropy and nearly dense geodesics---as shown in \cite{BW}, \cite{ConP}, and \cite{KW}. Geodesic flows with complex dynamics on spheres were already known from the work of Donnay \cite{Don1,Don2}, whose arguments relied on the presence of negative curvature. In contrast, the other constructions mentioned above were achieved even within the space of positively curved metrics, demonstrating that complicated dynamics for the geodesic flow do not require the presence of negative sectional curvature.

Another way to relax the Anosov condition is to consider \textit{partially hyperbolic systems} (see Subsection \ref{Preliminares Partial Hyperbolicity}), which still exhibit contracting and expanding directions, but these are no longer complementary. In addition, they possess \textit{center directions} along the orbits where the dynamics may be neutral (neither contracting nor expanding), making the analysis considerably more delicate.

In the setting of geodesic flows, the first examples of partially hyperbolic geodesic flows were constructed by Carneiro and Pujals in \cite{CP}. Recently, we extended their results in a subsequent work \cite{dJPR}, where we produced metrics whose geodesic flows display rich dynamical properties, including ergodicity with respect to the Liouville measure (and hence topological transitivity), uniqueness of the measure of maximal entropy, and expansivity. All these properties occur for metrics that lie outside the class of Anosov geodesic flows, but on the boundary of metrics without conjugate points (see \cite{Rug0}).

In the present work, we further extend the properties of such examples by establishing their robust transitivity, which leads to our first main result:

\begin{theoalph}
\label{THM A}
    There exists a manifold $M$ that admits a $C^2$-open set of Riemannian metrics with conjugate points and topologically transitive geodesic flow.
\end{theoalph}

To the best of our knowledge, this is the first result establishing the existence of robustly transitive geodesic flows outside the class of Anosov metrics. This partially addresses the questions posed above and also provides a complete answer to a question left open by Carneiro and Pujals \cite[Question 1]{CP} regarding the existence of transitive partially hyperbolic geodesic flows for metrics that admit conjugate points. 

The study of robustly transitive dynamical systems beyond uniform hyperbolicity dates back to the 70s with the work of Mañé \cite{Mane2}, and was further developed by Bonatti and Díaz \cite{BD}. A common feature of their examples is that they are partially hyperbolic systems. 
Later, Bonatti, Díaz and Pujals characterized robustly transitive diffeomorphisms, proving that every $C^1$-robustly transitive diffeomorphism must admit a dominated splitting---a necessary condition for partial hyperbolicity. For flows, a similar result for the Poincaré map was obtained by Bonatti, Gourmelon, and Vivier \cite{BGV}.

In our setting, we say that a Riemannian metric $g$ is robustly transitive if it admits a $C^2$-neighbourhood $\mathcal{U}$ such that for every $\tilde{g}\in\mathcal{U}$, the geodesic flow of $\tilde{g}$ is topologically transitive. In light of the aforementioned results, it is natural to conjecture that an analogous characterization should hold in the context of geodesic flows, although the perturbation techniques required would be of a fundamentally different nature:

\begin{conj}
    Let $g$ be a robustly transitive Riemannian metric. Then, its geodesic flow must admit a dominated splitting.
\end{conj}
More generally,
\begin{conj}
    Let $X$ be a robustly transitive contact flow (inside the class of contact flows); then it must admit a dominated splitting.
\end{conj}

Our result provides the first evidence supporting the above conjectures. Moreover, by a result of Contreras \cite{Cont}, a dominated splitting for the geodesic flow (more generally contact flows) already implies partial hyperbolicity.

The proof of Theorem~\ref{THM A} combines three main ingredients. First, the construction in \cite{dJPR} furnishes the required examples and several of their essential properties. Second, we introduce a key property in the context of flows, called the \textit{SH-saddle} property, where ``\textit{SH}'' stands for ``\textit{Some Hyperbolicity}''. This notion was originally introduced by Pujals and Sambarino \cite{PS} for partially hyperbolic diffeomorphisms as a criterion to obtain $C^1$-robust minimality of the stable foliation. It was later extended by the second author \cite{Pi2} to a broader class of partially hyperbolic systems---including symplectic diffeomorphisms---yielding new examples of $C^1$-robustly transitive, non-Anosov partially hyperbolic diffeomorphisms on $\mathbb{T}^n$.

The \textit{SH-saddle} property requires the existence of a uniform constant $L>0$ such that every strong stable (resp. strong unstable) leaf of length at least $L$ contains a point whose backward (resp. forward) iterates exhibit uniform expansion in some center direction (see Section~\ref{sectionshproperty} for the precise definition). 

The third ingredient is a general criterion for robust transitivity of partially hyperbolic Riemannian metrics, which we prove in this work. Let us first fix some notation: for a partially hyperbolic flow $\varphi^t: M\rightarrow M$, we are going to denote its $D\varphi^t$-invariant splitting by $TM=E^{ss}\oplus E^c\oplus \langle X\rangle\oplus E^{uu}$, where $E^c$ is the center bundle, $E^{ss}$ and $E^{uu}$ are, respectively, the strong-stable and strong-unstable bundles. Under this notation, we present the following result:
\begin{theoalph}
\label{criterion of robust transitivity}
    Let $g_1$ be a $C^2$ Riemannian metric on a compact differentiable manifold $M$ of dimension $n$ with no conjugate points and let $\varphi^t:T^1M\to T^1M$ be its geodesic flow. Suppose that $\varphi^t$ is expansive and partially hyperbolic with the SH-Saddle property of index $(d,d)$, where $d=n-1-\mathrm{dim} E^{ss}=n-1-\mathrm{dim}E^{uu}$. Then there is $\cU$ a $C^2$-neighborhood of $g_1$ such that if $g\in \cU$, then the geodesic flow of $g$ is topologically transitive.
\end{theoalph}

One may ask whether the hypotheses are sharp. Up to the present moment, we do not know how to relax any of the hypotheses. If the above conjectures hold, then one can expect to remove expansivity or non-existence of conjugate points. Unlike partial hyperbolicity and the \textit{SH-saddle} property, none of those properties are $C^2$-open (in general) in the set of Riemannian metrics, and they are, respectively, topological and geometric features of the systems, not dynamical. In fact, if $g_1$ is a Riemannian metric with no conjugate points that is robustly expansive in the set of Hamiltonian flows, then it is proved in \cite{LRR} that the geodesic flow of $g_1$ is Anosov. Hence, in this case, it is already known that robustly there are no conjugate points and the geodesic flow is robustly transitive inside this class. In our case, those assumptions will be important to obtain the local product structure proved by Ruggiero in \cite{Rug1}, and this property in general is not known to be robust.


\subsection{Strategy of the proofs}

Altogether, the proof of Theorem \ref{THM A} goes as follows: the examples in \cite{dJPR} are partially hyperbolic, not Anosov, expansive, topologically transitive, and the metrics have no conjugate points. We establish the \textit{SH-saddle} property for the examples and prove that this is a $C^2$-open property in the set of Riemannian metrics with partially hyperbolic geodesic flow. We can then apply Theorem \ref{criterion of robust transitivity} to find a $C^2$-neighbourhood consisting of Riemannian metrics with topologically transitive geodesic flows, call it $\mathcal{U}$. By a result by Ruggiero \cite{Rug0}, the metrics given by the examples must lie in the boundary of the set of metrics with no conjugate points; thus, we find a smaller $C^2$-open set inside $\mathcal{U}$ of Riemannian metrics with topologically transitive geodesic flow and conjugate points, and the proof is complete. 

The proof of Theorem \ref{criterion of robust transitivity} is the most delicate part of this article. The arguments rely heavily on the dynamics and topological aspects of the setting rather than the Riemannian geometry. Moreover, we will use that sufficiently close to the initial metric, some properties persist, then we will restrict ourselves to smaller and smaller neighborhoods whenever necessary. Let us describe the steps of the proof:
\begin{itemize}
    \item[\textbf{Step 1}:] Let $g_1$ be a Riemannian metric and $\varphi^t$ as in the hypothesis. Expansivity and non-existence of conjugate points guarantee, by \cite{Rug1}, the existence of two continuous foliations with $C^0$ leaves given by the stable and unstable sets. Moreover, this pair of foliations presents a local product structure. This implies transitivity for $\varphi^t$.
    \item[\textbf{Step 2}:] From \cite{Rug0}, there exists a $C^2$-neighbourhood of $g_1$, say $\mathcal{U}_0$, such that if $g\in \mathcal{U}_0$, then its geodesic flow $\psi^t$ is semi-conjugated to $\varphi^t$ by a surjective continuous function $h$. In this context, $h$ can be chosen as $C^0$-close to the identity as needed by shrinking $\mathcal{U}_0$. Notice that, in general, this is not sufficient to conclude the topological transitivity of $\psi^t$.
    \item[\textbf{Step 3}:] We now focus on the second part of the hypothesis, the partial hyperbolicity. It is well known that this is a $C^2$-open property in the space of Riemannian metrics; thus, we can choose a $C^2$-neighbourhood of $g_1$, say $\mathcal{U}_1$, consisting of Riemannian metrics with partially hyperbolic geodesic flows. Moreover, in this context, we can follow the ideas of \cite{Pi2} and introduce the notion of \textit{SH-saddle} for partially hyperbolic geodesic flows. We then show that this is an open property in the $C^2$ topology for Riemannian metrics. Hence, we choose another $C^2$-neighborhood of $g_1$, say $\mathcal{U}_2\subset \mathcal{U}_1$, of Riemannian metrics with partially hyperbolic geodesic flows with the \textit{SH-saddle} property.
    \item[\textbf{Step 4}:] We prove that if $g\in \mathcal{U}:=\mathcal{U}_0\cap \mathcal{U}_1\cap \mathcal{U}_2$, then its geodesic flow $\psi^t$ is transitive. To this end, we show that given two open sets $U$ and $V$, they are mixed by the dynamics of $\psi^t$. First, the \textit{SH-saddle} property allows us to ``grow'', inside $U$ and $V$, topological disks tangent (in some sense) to the center-stable and center-unstable directions. By applying the semi-conjugation, these topological disks cannot collapse, although they may bend and twist a little bit. Then we can project them to some stable and unstable sets for the initial flow $\varphi^t$. We can show that their projections contain open sets inside these $C^0$-leaves. The transitivity of $\varphi^t$ mixes these open sets, and by adjusting the parameters $t$, we can conclude that there exists $T>0$ such that $\psi^T(U)\cap V\neq \emptyset$ and establish the topological transitivity of $\psi^t$.
    
\end{itemize}

\subsection{Organization of the paper}
In Section \ref{sphgeodflow}, we cover the necessary background in expansive geodesic flows and partial hyperbolicity. Section \ref{sectionshproperty} is devoted to introducing the notion of \textit{SH-saddle} property for partially hyperbolic flows and to proving some necessary results about this property. In Section \ref{Section Robust transitivity criterion}, we prove Theorem \ref{criterion of robust transitivity}. Finally, in Section \ref{proof of main theorem}, we show the examples given by \cite{dJPR} and \cite{CP} satisfy the \textit{SH-saddle} property and prove Theorem \ref{THM A}.

\section{Preliminaries} \label{sphgeodflow}

\subsection{Expansive geodesic flows}\label{ssexpandts}

Let $\varphi^t:X \to X$ be a continuous flow on a metric space $(X,d)$. The flow $\varphi^t$ is said to be \textit{expansive} if there exists a constant $\chi>0$ such that for every $x\in X$ we have the following property: if for a given $y\in X$ there exists a continuous and surjective map $r_y:\R\to \R$ with $r_y(0)=0$ such that $$d(\varphi^t(x),\varphi^{r_y(t)}(y))\leq \chi \ \ \text{for every} \ \ t\in \R,$$ then there exists $t_0\in \R$ such that $\varphi^{t_0}(x)=y$. We call $\chi$ the \textit{expansivity constant}. In other words, every two different orbits of an $\chi$-expansive flow are $\chi$-separated eventually in time.

We say that a continuous flow $\varphi^t:X \to X$ is \textit{topologically stable} if there exists a $C^0$ neighbourhood $\mathcal{V}$ of $\varphi^t$ such that, for every flow $\psi^t\in \mathcal{V}$ there are continuous and surjective functions $h:X \to X$ and $r:X \times \R \to \R$ with $r(\cdot ,0)=0$ such that $$h \circ \psi^t(x)=\varphi^{r(x,t)}\circ h(x) \ \ \text{for every} \ \ t\in \R, \ x\in X.$$ 

In dynamical systems, the \textit{stable} and \textit{unstable} sets of a point play a key role. For a flow $\varphi^t:X \to X$ we define these sets by:
\begin{eqnarray*}
	W^s_{\varphi}(x)=\{ y \in X: \lim_{t \to+\infty} d(\varphi^t(x),\varphi^t(y)) =0 \}, \\
	W^u_{\varphi}(x)=\{ y \in X: \lim_{t \to-\infty}d(\varphi^t(x),\varphi^t(y))=0\}.
\end{eqnarray*} 

From now on we are going to focus on geodesic flows, i.e., $\varphi^t$ will be the geodesic flow associated to a Riemannian manifold $(M,g)$ and the metric space $X$ will be the unit tangent bundle $T^1M=\{v\in TM: g(v,v)=1\}$. We remark that the unit tangent bundle depends on the metric, but two different Riemannian metrics have diffeomorphic unit tangent bundles. Thus, for consistency, we keep the notation $T^1M$ for all these unit tangent bundles

In hyperbolic dynamics, a key property is that the stable and unstable sets are in fact, differentiable manifolds with local product structure, and moreover, this local behaviour implies strong dynamic consequences. For expansive geodesic flows of Riemannian manifolds with no conjugate points, we have an analogous result due to R. O. Ruggiero. We recall that a Riemannian manifold has no conjugate points if the exponential map is non-singular at every point. 

\begin{thm}[Theorems 1 and 2 in \cite{Rug1}] \label{teoLPSruggiero}
	Let $(M,g_1)$ be a $C^{\infty}$ compact Riemannian manifold of dimension $n$ with no conjugate points. Let $\varphi^t:T^1M\to T^1M$ be the geodesic flow on the unit tangent bundle and assume that $\varphi^t$ is $\chi$-expansive. Then the following properties hold:
	\begin{enumerate}
        \item For every point $\theta\in T^1M$ the sets $W^s_{\varphi}(\theta)$ and $W^u_{\varphi}(\theta)$ are $C^0$ sub-manifolds of dimension $n-1$. Moreover, $W^s_{\varphi}$ and $W^u_{\varphi}$ give continuous foliations of $T^1M$ which induce a local product structure.
        \item The set of closed orbits is dense in $T^1M$.
        \item The flow $\varphi^t$ is topologically transitive.
	\end{enumerate}
\end{thm}

By local product structure we mean the following: for every $\theta \in T^1M$ there is a local transverse section $\Sigma_\theta=\exp_{\theta} \left\{w\in T_{\theta}T^1M: \norm{w}<\chi, \ g(w,X(\theta))=0\right\}$, where $X(\theta)$ is the direction of the flow, and a homeomorphism $F:(-1,1)^{2(n-1)}\to \Sigma_\theta$ such that:
\begin{enumerate}
	\item $F((-1,1)^{n-1}\times \{y_0\})$ is a subset of the connected component of 
	$$\bigcup_{t\in \R} \varphi^t(W^s_{\varphi}(F(0,y_0)))\cap \Sigma_\theta 
	$$ containing $F(0,y_0)$ for every $y_0\in (-1,1)^{n-1}$.
	
	\item $F(\{x_0\}\times (-1,1)^{n-1})$ is a subset of the connected component of 
	$$\bigcup_{t\in \R} \varphi^t(W^u_{\varphi}(F(x_0,0)))\cap \Sigma_\theta
	$$ containing $F(x_0,0)$ for every $x_0\in (-1,1)^{n-1}$.
\end{enumerate}
The sets 
$$W^{cs}_{\varphi}(\theta)=\bigcup_{t\in \R}\varphi^t(W^s_{\varphi}(\theta)) \ \ \text{and} \ \ W^{cu}_{\varphi}(\theta)=\bigcup_{t\in \R}\varphi^t(W^u_{\varphi}(\theta))
$$ are called the \textit{center stable} and \textit{center unstable} sets of $\theta$ respectively. 

\begin{rem}\label{local unstable}
By the continuity of unstable manifolds and the local product structure, for a fixed \(\epsilon > 0\), the local unstable manifolds
\[
W^u_{\epsilon}(z) = \left\{ w \in T^1M \;\middle|\; w\in W^u(z),\; \sup_{t \le 0} d(\varphi^t(z), \varphi^t(w)) \le \epsilon \right\}
\]
are locally continuous sets. The local stable manifolds are defined analogously.
\end{rem}



For diffeomorphisms with uniformly hyperbolic behavior, the stable and unstable sets with local product structure are enough to have shadowing properties and therefore topological stability. Since for expansive geodesic flows with no conjugate points we also have stable and unstable sets with local product structure, the same shadowing lemma holds for this kind of geodesic flows, and in consequence, we have the following stability result. The proof once again is due to R. O. Ruggiero.




\begin{thm}[Theorem 3.2 in \cite{Rug2}] \label{TeoRugtopstab}
	Let $\varphi^t:T^1M \to T^1M$ be a geodesic flow of a compact, $n$-dimensional manifold $(M,g_1)$ with no conjugate points. If $\varphi^t$ is expansive, then it is topologically stable. 
\end{thm}

From now on, we fix a compact Riemannian manifold $(M,g_1)$ of dimension $n$ with no conjugate points such that its geodesic flow $\varphi^t:T^1M \to T^1M$ is $\chi$-expansive. We also fix the $C^0$ neighbourhood $\mathcal{V}_\chi$ of $\varphi^t$ from Theorem \ref{TeoRugtopstab}. Then, for every flow $\psi^t\in \mathcal{V}_\chi$ there are continuous and surjective functions $h_{\psi}:T^1M \to T^1M$ and $r_{\psi}:T^1M \times \R \to \R$ with $r_{\psi}(x,0)=0$ such that: 
\begin{equation}\label{conjrelation}
	h_{\psi} \circ \psi^t(\theta)=\varphi^{r_{\psi}(\theta,t)}\circ h_{\psi}(\theta) \ \ \text{for every} \ \ t\in \R, \ \theta\in T^1M.
\end{equation} 

Notice that topological stability is a weaker notion than structural stability. The problem is that the function $h_{\psi}$ mentioned above is continuous and surjective, but in most cases will not be injective. Nevertheless, if we are sufficiently close to $\varphi^t$, the fibers of the map $h_{\psi}$ will have small size, and this will be enough for our purposes. Here by \textit{fiber} of $\theta \in T^1M$ we mean the set $h_{\psi}^{-1}(h_{\psi}(\theta))$. In particular, the map $h_{\psi}$ is injective if and only if $h_{\psi}^{-1}(h_{\psi}(\theta))=\{ \theta\}$ for every $\theta \in T^1M$ (the fibers are trivial).

The function $h_{\psi}$ is given by the Shadowing Lemma above and depends continuously on $\psi^t$. In fact, the closer $\psi^t$ is to $\varphi^t$, the smaller the $C^0$ distance between $h_{\psi}$ and $\textit{Id}$. This implies the following lemma.



\begin{lem} \label{controlfibras}
For every $\delta>0$ there is a $C^0$ neighborhood $\mathcal{V}_\delta$ of $\varphi^t$ such that, for every flow $\psi^t \in \mathcal{V}_\delta$ we have $\operatorname{diam}(h_{\psi}^{-1}(h_{\psi}(\theta)))<\delta$ for every $\theta \in T^1M$. 
\end{lem}

\begin{proof}
	Given $\delta>0$ take a small neighborhood $\mathcal{V}$ of $\varphi^t$ such that for every $\psi^t \in \mathcal{V}$ we have $d_{C^0}(h_{\psi},\textit{Id})<\delta/4$. Then, if $\theta\in T^1M$ and $\eta, \xi \in h_{\psi}^{-1}(h_{\psi}(\theta))$ by the triangular inequality, we have
$$d(\theta,\eta) \leq  d(\theta, h_{\psi}(\theta))+d(h_{\psi}(\theta),h_{\psi}(\eta))+d(h_{\psi}(\eta),\eta) < \frac{\delta}{4}+0+\frac{\delta}{4} = \frac{\delta}{2}.$$
In the same way $d(\theta,\xi)<\delta/2$. Then by the triangular inequality $d(\eta,\xi)<\delta$.
\end{proof}

The previous lemma is going to be important for us during the technical part of the argument in the proof of Theorem \ref{criterion of robust transitivity}. For that, we will need to choose $\delta>0$ much smaller than some uniform constants  (see Lemma \ref{remove h}).

In the same way as we mentioned above, the function $r_{\psi}:T^1M \times \R \to \R$ also varies continuously with the flow $\psi^t$. 



\subsection{Uniform hyperbolicity} \label{sAnosov}
We say that a diffeomorphism $f:M \to M$ is \textit{Anosov} if there exists a nontrivial $Df$-invariant splitting $TM=E_f^{s} \oplus E_f^{u}$ of the tangent bundle and constants $\lambda _s, \lambda_u$ such that: 
\begin{equation*}
	\|D f_x|_{E^{s}_f}\|< \lambda_s <1< \lambda_u   <\  \|D f_x|_{E^{u}_f}\|. 
\end{equation*}
It is well known that the bundles $E_f^{s}$ and $E_f^{u}$ integrate into unique $f$-invariant foliations $\cW^{s}_f$ and $\cW^{u}_f$ respectively, called the \textit{stable} and \textit{unstable} foliations \cite{HPS}. 

In the same way, we say that a flow $\varphi^t:M\to M$ generated by a vector field $X:M\to TM$ is Anosov if the tangent bundle $TM$ splits into $D\varphi$-invariant continuous subbundles $TM=E^{s}_{\varphi}\oplus \langle X \rangle \oplus E^{u}_{\varphi}$ such that 
\begin{equation*}
	\norm{D\varphi^t(v^{ss})}<\lambda_{ss}^{t} \ \quad \text{and} \quad \lambda_{uu}^t< \norm{D\varphi^t(v^{uu})} \ \ \text{for} \ t>0
\end{equation*}
for some Riemannian metric $\norm{\cdot}$ and some $\lambda_{s}<1<\lambda_{u}$ and all unit vectors $v^{s}\in E^{s}_{\varphi}$ and $v^{u}\in E^{u}_{\varphi}$. 
Finally, we say that a Riemannian metric $(M,g)$ is Anosov if its corresponding geodesic flow $\varphi^t: T^1M\to T^1M$ is Anosov. Notice that this property is robust in the $C^2$ topology on the space of Riemannian metrics. 

It is very easy to see that an Anosov flow is an expansive flow. Moreover, a classical result states that if a Riemannian metric $g$ has negative sectional curvature at every point, then its geodesic flow is an Anosov flow \cite{An}, and therefore an expansive geodesic flow. The converse is not true, as our example in \cite{dJPR} shows.

\subsection{Partial hyperbolicity}
\label{Preliminares Partial Hyperbolicity}
We say that a diffeomorphism $f:M \to M$ is \textit{partially hyperbolic} if there exists a nontrivial $Df$-invariant splitting $TM=E_f^{ss} \oplus E_f^c \oplus E_f^{uu}$ of the tangent bundle and continuous functions $\lambda _s,\lambda _c^-, \lambda _c^+, \lambda_u$ with $\lambda_s <1< \lambda_u$ and $\lambda_s < \lambda_c^- < \lambda_c^+< \lambda_u$ such that: 
\begin{equation*}
	\|D f_x|_{E^{ss}_f}\|<\lambda_s, \qquad  \lambda_c^- <\  \|D f_x|_{E^c_f}\|< \lambda_c^+ , \qquad  \lambda_u   <\  \|D f_x|_{E^{uu}_f}\|. 
\end{equation*}
We will denote by $\cPH(M)$ the set of all partially hyperbolic diffeomorphisms on $M$. Like in the Anosov case, it is well known that the strong bundles $E_f^{uu}$ and $E_f^{ss}$ integrate into unique $f$-invariant foliations $\cW^{uu}_f$ and $\cW^{ss}_f$ respectively, called the \textit{strong unstable} and \textit{strong stable} foliations \cite{HPS}. For $* = uu,ss$, and  for any $x \in M$, we denote by $\cW_f^*(x)$ the leaf of $\cW_f^*$ through $x$. In the following, for any $*\in \{ss,uu\}$, we denote by $d_{\mathcal{W}_f^*}$ the leaf-wise distance, and for any $x \in M$ and for any  $\varepsilon>0$, we denote by 
$$\cW_f^*(x,\varepsilon):=\{y \in \cW_f^*(x): d_{\cW_f^*}(x,y)< \varepsilon\}$$ the $\varepsilon$-ball in $\cW_f^*$ of center $x$ and radius $\varepsilon$.

In the same way, we say that a flow $\varphi^t:M\to M$ generated by a vector field $X:M\to TM$ is partially hyperbolic if the tangent bundle $TM$ splits into $D\varphi$-invariant continuous subbundles $TM=E^{ss}_{\varphi}\oplus E^{c}_{\varphi}\oplus \langle X \rangle \oplus E^{uu}_{\varphi}$ such that 
\begin{equation*}
	\norm{D\varphi^t(v^{ss})}<\lambda_{ss}^{t}<\norm{D\varphi^t(v^c)} < \lambda_{uu}^t< \norm{D\varphi^t(v^{uu})} \ \ \text{for} \ t>0
\end{equation*}
for some Riemannian metric $\norm{\cdot}$ and some $\lambda_{ss}<1<\lambda_{uu}$ and all unit vectors $v^{ss}\in E^{ss}_{\varphi}$, $v^c\in E^c_{\varphi}$ and $v^{uu}\in E^{uu}_{\varphi}$. 
Like in the discrete case, the strong bundles $E^{ss}_{\varphi}$ and $E^{uu}_{\varphi}$ integrate to foliations $\cW^{uu}_{\varphi}$ and $\cW^{ss}_{\varphi}$, which are invariant by the flow $\varphi^t$.

Notice that if $\varphi^t:M \to M$ is a partially hyperbolic flow with a splitting of the form $TM=E^{ss}_{\varphi}\oplus E^{c}_{\varphi}\oplus \langle X \rangle \oplus E^{uu}_{\varphi}$, then for every $T\in \R$ the map $f:=\varphi^T:M\to M$ is a partially hyperbolic diffeomorphism with a splitting $TM=E_f^{ss} \oplus E_f^c \oplus E_f^{uu}$, where $E^{ss}_f=E^{ss}_{\varphi}$, $E^{uu}_f=E^{uu}_{\varphi}$ and $E^c_f=E^c_{\varphi}\oplus \langle X \rangle$. In particular, $\text{dim}E^c_f=\text{dim}E^c_{\varphi}+1$, since the direction of the flow becomes part of the center bundle of $f$.

Finally, we say that a Riemannian metric $(M,g)$ is partially hyperbolic if its corresponding geodesic flow $\varphi^t: T^1M\to T^1M$ is partially hyperbolic. Notice that this property is robust in the $C^2$ topology on $g$.

\section{SH-Saddle property} \label{sectionshproperty} 


In this section, we are going to introduce the definition of the SH-Saddle property for flows (in fact, for geodesic flows). We begin by presenting the definition of the SH-Saddle property for diffeomorphisms as well as some results given in \cite{Pi2}. Then we translate these definitions and results to the flow setting.

\subsection{SH-Saddle property for diffeomorphisms}

Let $V$ be a $\mathbb{R}$-vector space with an inner product. A \textit{cone} in $V$ is a subset $\cC$ such that there is a non-degenerate quadratic form $B:V \to \mathbb{R}$ such that $$\cC=\{v\in V: B(v)\leq 0\}.$$ Analogously we can express the cone $\cC$ according to a decomposition $V=E\oplus F$: 
\begin{equation} \label{eqcono1}
	\cC=\{v=(v_E,v_F): \norm{v_E}\leq \mu \norm{v_{F}}\}
\end{equation} for some $\mu>0$. In this case we observe that $B(v)=-\mu^2\norm{v_{F}}^2+\norm{v_{E}}^2$. We are going to say that the number $\mu$ in Equation \eqref{eqcono1} is the \textit{size} of the cone. In some cases, we will note by $\cC_{\mu}$ instead of $\cC$ to make emphasis on the size of $\cC$. The \textit{dimension of a cone} is the maximal dimension of any subspace contained in the cone. 

Now let $f$ be a partially hyperbolic diffeomorphism with a splitting of the form $TM=E_f^{ss} \oplus E_f^c \oplus E_f^{uu}$. We are going to note by $k=\text{dim}E^c_f$. Then a \textit{$d$-center cone} in $x\in M$ is simply a cone $\cC(x)$ in $E^{c}_f(x)$ of dimension $d\leq k$. Recall that 
$$\cW_f^*(x,\varepsilon):=\{y \in \cW_f^*(x): d_{\cW_f^*}(x,y)< \varepsilon\}$$ is the $\varepsilon$-ball in $\cW_f^*$ of center $x$ and radius $\varepsilon$ for  $*\in \{ss,uu\}$. 

\begin{df}[SH-Saddle property for unstable foliations] \label{defshuu}
	Given $f\in \cPH(M)$ we say that the strong unstable foliation $\cW^{uu}_f$ has the SH-Saddle property of index $d\leq k$ if there are constants $L>0$, $\mu>0$, $\lambda_0>1$ and $C>0$ such that the following hold. 
	
	For every point $x\in M$, there is a point $x^u\in \cW^{uu}_f(x,L)$ such that:
	\begin{enumerate}
		\item \label{sh1} There is a $d$-center cone field of size $\mu$ along the forward orbit of $x^u$ which is $Df$-invariant, i.e. there exist $\cC^u_{\mu}(f^l(x^u))\subset E^c_f(f^l(x^u))$ such that $Df(\cC^u_{\mu}(f^l(x^u)))\subset \cC^u_{\mu}(f^{l+1}(x^u))$ for every $l\geq 0$. \label{c1shu}
		\item $\norm{Df^n_{f^l(x^u)}(v)}\geq C\lambda_0^n\norm{v}$  for every $v\in \cC^u_{\mu}(f^{l}(x^u))$ and every $l,n\geq0$. \label{c2shu}
	\end{enumerate}
\end{df}

We can make an analogous definition for the strong stable foliation. In this case, we ask for the invariance of the cones for the past.
\begin{df}[SH-Saddle property for stable foliations] \label{defshss}
	Given $f\in \cPH(M)$ we say that the strong stable foliation $\cW^{ss}_f$ has the SH-Saddle property of index $d\leq k$ if there are constants $L>0$, $\mu>0$, $\lambda_{0}>1$ and $C>0$ such that the following hold. 
	
	For every point $x\in M$, there is a point $x^s\in \cW^{ss}_f(x,L)$ such that: 
	\begin{enumerate}
		\item There is a $d$-center cone field of size $\mu$ along the backward orbit of $x^s$ which is $Df^{-1}$-invariant, i.e. there exist $\cC^s_{\mu}(f^l(x^s))$ such that $Df^{-1}(\cC^s_{\mu}(f^l(x^s)))\subset \cC^s_{\mu}(f^{l-1}(x^s))$  for every $l \leq 0$. \label{c1shs}
		\item $\norm{Df^n_{f^l(x^s)}(v)}\geq C\lambda_0^{-n}\norm{v}$  for every $v\in \cC^s_{\mu}(f^l(x^s)))$ and every $l,n\leq 0$. \label{c2shs}
	\end{enumerate}
\end{df}
\begin{df}[SH-Saddle property] \label{defshf}
	We say that $f\in \cPH(M)$ has $(d_1,d_2)$ SH-Saddle property if the following conditions hold:
	\begin{enumerate} 
		\item $\cW^{ss}_f$ has the SH-Saddle property of index $d_1$.
		\item $\cW^{uu}_f$ has the SH-Saddle property of index $d_2$.
	\end{enumerate}
\end{df}	

Notice that in general we do not necessarily have $d_1+d_2=k$. In fact, in many cases, we are going to have $d_1+d_2<k$. In some parts of the article, for simplicity and when it is not needed, we are going to omit the indexes $(d_1,d_2)$ and we are just going to say that a partially hyperbolic diffeomorphism has the SH-Saddle property. 

We also remark that the SH-Saddle property does not depend on the choice of the Riemannian metric. Therefore, we get the following proposition.

\begin{prop} \label{fshifffksh}
	A partially hyperbolic diffeomorphism $f$ has the SH-Saddle property if and only if $f^N$ has the SH-Saddle property for some $N\in \N$.
\end{prop}	

Let us introduce some notation that will be useful throughout the article and will help us to get a better understanding of what the SH-Saddle property means. Let $f\in \cPH(M)$ be such that its unstable foliation has the SH-Saddle property of index $d\leq \text{dim}E^c_f$ and let $L>0$, $\mu>0$, $\lambda_0>1$ and $C>0$ be the constants given by Definition \ref{defshuu}. We can define the following subset:
\begin{equation*} \label{defH+}
	H^+_{\lambda_0,d}(f)=\{x^u\in M: \text{conditions (\ref{c1shu}) and (\ref{c2shu}) of Definition \ref{defshuu} are satisfied}\}
\end{equation*}  
Then the unstable foliation has the SH-Saddle property of index $d$ if and only if $$H^+_{\lambda_0,d}(f)\cap \cW^{uu}_f(x,L)\neq \emptyset \  \ \text{for every} \ x\in M.$$ In the same way, let $f \in \cPH(M)$ be such that its stable foliation has the SH-Saddle property of index $d$ and let $L>0$, $\mu>0$, $\lambda_0>1$ and $C>0$ be the constants given by Definition \ref{defshss}, then we can define the following subset:
\begin{equation*}\label{defH-}
	H^{-}_{\lambda_0,d}(f)=\{x^s\in M: \text{conditions (\ref{c1shs}) and (\ref{c2shs}) of Definition \ref{defshss} are satisfied}\}
\end{equation*} 
and the stable foliation has the SH-Saddle property of index $d$ if and only if  $$H^-_{\lambda_0,d}(f)\cap \cW^{ss}_f(x,L)\neq \emptyset \ \ \text{for every} \ x\in M.$$ 

\begin{rem}
	The sets $H^{*}_{\lambda_0,d}(f)$ are closed subsets of $M$, for $*=+,-$.
\end{rem} 

The following is the main property of the SH-Saddle condition: it says that having an unstable foliation with the SH-Saddle property is a $C^1$-open property among $\cPH(M)$. 

\begin{thm}[Theorem 2.7 in \cite{Pi2}] \label{shisopen}
	Suppose that the unstable foliation of $f\in \cPH(M)$ has the SH-Saddle property of index $d$. Then there are constants $\lambda>1$, $L>0$ and a  $C^1$-neighbourhood $\mathcal{V}$ of $f$ such that, if $g\in \mathcal{V}$ then
	$H^+_{\lambda,d}(g)\cap \cW^{uu}_g(x,L)\neq \emptyset$ for every $x\in M$ (i.e.: the unstable foliation $\cW^{uu}_g$ has the SH-Saddle property of index $d$ with constants $\lambda>1$ and $L>0$). 
\end{thm}
	
Since the $C^1$-openness of the SH-Saddle property for stable manifolds is completely analogous, we get the following corollary.
	
\begin{cor}
	The SH-Saddle property is $C^1$-open among $\cPH(M)$. 
\end{cor} 

The following corollary from Theorem \ref{shisopen} will be useful in the proof of the robust transitivity criterion. First, let us say that $D$ is a center disk of dimension $d\leq \text{dim}E^c_f$ if it is a $d$-dimensional embedded disk contained in some center plaque.
	
\begin{cor}[Corollary 2.9 in \cite{Pi2}] \label{corotamanodisco}
	Let  $f\in \cPH(M)$ be such that its unstable foliation has the SH-Saddle property of index $d$ and let $\lambda>1$, $\delta_1>0$ and $\mathcal{V}$ as in Theorem \ref{shisopen}. Take $g\in \mathcal{V}$, $x^u \in H^+_{\lambda,d}(g)$ and $\cD^u$ a center disk of dimension $d$ tangent to $\cC^u(x^u)$. Then there is $N>0$ such that $g^n(\cD^u)$ contains a center disk of dimension $d$, centered at $g^n(x^u)$ of diameter bigger than $2\delta_1$ for every $n\geq N$. 
		
	Analogously with the stable foliation. 
\end{cor}




\subsection{SH-Saddle property for flows} \label{subsectionSHflows}
	
Recall that a flow $\varphi^t:M\to M$ generated by a vector field $X:M\to TM$ is partially hyperbolic if the tangent bundle $TM$ splits into $D\varphi$-invariant continuous subbundles $TM=E^{ss}_{\varphi}\oplus E^{c}_{\varphi}\oplus \langle X \rangle \oplus E^{uu}_{\varphi}$ such that 
	\begin{equation*}
		\norm{D\varphi^t(v^{ss})}<\lambda_{ss}^{t}<\norm{D\varphi^t(v^c)} < \lambda_{uu}^t< \norm{D\varphi^t(v^{uu})} \ \ \text{for} \ t>0
	\end{equation*}
	for some Riemannian metric $\norm{\cdot}$ and some $\lambda_{ss}<1<\lambda_{uu}$ and all unit vectors $v^{ss}\in E^{ss}_{\varphi}$, $v^c\in E^c_{\varphi}$ and $v^{uu}\in E^{uu}_{\varphi}$.

	\begin{df}[SH-Saddle property for unstable foliations of flows] Given a partially hyperbolic flow $\varphi^t:M \to M$, we say that the strong unstable foliation $\cW^{uu}_{\varphi}$ has SH-Saddle property of index $d\leq c$ if there exist $T\in \R$ such that the induced partially hyperbolic diffeomorphism $f=\varphi^T$ has strong unstable foliation $\cW^{uu}_f$ with SH-Saddle property of index $d$. Analogously, for the strong stable foliation.
	\end{df} 
	
	In addition, we have the SH-Saddle property for flows.
	
	\begin{df}[SH-Saddle property for flows] We say that a partially hyperbolic flow $\varphi^t:M\to M$ has the SH-Saddle property of index $(d_1,d_2)$ if there is $T\in \R$ such that $f=\varphi^T$ has $(d_1,d_2)$ SH-Saddle property as a partially hyperbolic diffeomorphism.
	\end{df}

	\begin{rem}
		Recall that if $\varphi^t:M\to M$ is a partially hyperbolic flow with $\textnormal{dim}E^c_\varphi=k$ then $f:=\varphi^T:M \to M$ is a partially hyperbolic diffeomorphism with $\textnormal{dim}E^c_f=k+1$.
	\end{rem}
	
	Notice that if two flows $\varphi^t,\psi^t:M\to M$ are $C^1$ close, then their time $T$ maps $f=\varphi^T$ and $g=\psi^T$ are $C^1$ close diffeomorphisms. In particular, Theorem \ref{shisopen} implies that the SH-Saddle property is $C^1$-open among partially hyperbolic flows. We summarize this observation in the following proposition.
	
	\begin{prop} \label{shfflowsisopen}
		The SH-Saddle property is $C^1$-open among partially hyperbolic flows.
	\end{prop}
	\begin{proof}
		Take a partially hyperbolic flow $\varphi^t:M \to M$ with the SH-Saddle property. By definition, we have that there is $T\in \R$ such that the map $f=\varphi^T$ has the SH-Saddle property as a partially hyperbolic diffeomorphism. By Theorem \ref{shisopen}, there is a $C^1$ neighbourhood $\cB$ of $f$ such that every partially hyperbolic diffeomorphism $g \in \cB$ has the SH-Saddle property. Now just take a sufficiently $C^1$ small neighbourhood $\mathcal{V}$ of $\varphi^t$ such that for every flow $\psi^t \in \mathcal{V}$, we have $\psi^T \in \cB$. As a result, every flow $\psi^t \in \mathcal{V}$ has the SH-Saddle property.  
	\end{proof}
	
	Now, since every Riemannian metric has its corresponding geodesic flow, we can translate the definition of the SH-Saddle property to partially hyperbolic Riemannian metrics. Recall that a Riemannian metric is said to be partially hyperbolic if its corresponding geodesic flow is partially hyperbolic.
	
	\begin{df}[SH-Saddle property for Riemannian metrics]
		A $C^2$ partially hyperbolic Riemannian metric has the SH-Saddle property if its induced geodesic flow has the SH-Saddle property.
	\end{df}
	
	Notice that a $C^2$ small perturbation of the Riemannian metric implies a small perturbation on the geodesic field, and therefore a $C^1$ small perturbation on the flow. Then, by a similar argument as above, we get the following proposition.
	
	\begin{prop} \label{SHopenformetrics}
		The SH-Saddle property is a $C^2$-open property among $C^2$ partially hyperbolic Riemannian metrics.
	\end{prop}
	\begin{proof}
		Suppose $g_0$ is a $C^{\infty}$ Riemannian metric such that its geodesic flow $\varphi^t:T^1M\to T^1M$ is partially hyperbolic and has the SH-Saddle property. By Proposition \ref{shfflowsisopen} we know there is a $C^1$ neighbourhood $\mathcal{V}$ such that every flow $\psi^t \in \mathcal{V}$ has the SH-Saddle property. Then we just have to take $\cU$ a $C^2$ neighbourhood of $g_0$ in the space of $C^{\infty}$ Riemannian metrics such that for every metric $g\in \cU$, its geodesic flow $\varphi^t \in \mathcal{V}$.  	
	\end{proof}
	


	
To finish this section, consider a $C^2$ Riemannian metric $g$ on a compact differentiable manifold $M$ of dimension $n$ without conjugate points, and let $\varphi^t:T^1M\to T^1M$ be its geodesic flow. Suppose that $\varphi^t$ is expansive and partially hyperbolic. Then the symplectic structure on $T^1M$ imposes some conditions on the dimensions of the stable and unstable bundles and stable and unstable sets. 
	
	For instance, let $W^s_{\varphi}$ and $W^u_{\varphi}$ be the stable and unstable sets, respectively, given by the expansivity condition. Then it is easy to see that:
	$$ \text{dim}W^s_{\varphi}=\text{dim}W^u_{\varphi}=n-1.
	$$
	In addition, we have a partially hyperbolic splitting given by $$T(T^1M)=E^{ss}_{\varphi}\oplus E^c_{\varphi} \oplus \langle X \rangle \oplus E^{uu}_{\varphi}.$$ Then, in the same way as above, we have
	$$ \text{dim}E^{ss}_{\varphi}=\text{dim}E^{uu}_{\varphi}=:k 
	$$
	because the map $\pi:TM \to TM$ given by $\pi(x,v)=(x,-v)$ is a diffeomorphism. Then if $\varphi^t$ has the SH-saddle property of index $(d_1,d_2)$ where $d_1=\textnormal{dim}W^s_{\varphi}-\textnormal{dim}E^{ss}_{\varphi}$ and  $d_2=\textnormal{dim}W^u_{\varphi}-\textnormal{dim}E^{uu}_{\varphi}$, we necessarily have that:
	$$ d_1=d_2=n-1-k. 
	$$
	This observation will be useful in the proof of Theorem \ref{THM A}. 
	
\section{Robust transitivity criterion}
\label{Section Robust transitivity criterion}
In this section, we are going to give sufficient conditions for a $C^{\infty}$ partially hyperbolic Riemannian metric to be $C^2$ robustly transitive.
	
\subsection{Transitivity}
Recall that a diffeomorphism $f:M \to M$ is said to be \textit{transitive} if there is $x\in M$ such that $\overline{\cO^+ (f,x)}=M$. The following proposition gives an equivalent definition that is more manageable. 
	
\begin{prop} \label{propequivtrans}
	Given a diffeomorphism $f:M \to M$ the following are equivalent:
	\begin{itemize}
		\item $f$ is topologically transitive.
		\item For every pair of open sets $U$ and $V$ there is $N\in \mathbb{N}$ such that $f^N(U)\cap V \neq \emptyset$.
		\item There is a residual set $R$ such that $\omega(x)=M$, for every $x\in M$.
	\end{itemize}
\end{prop}

In the same way, we say that a flow $\varphi^t:M\to M$ is \textit{transitive} if there is a point $x\in M$ such that $\overline{\cO^+(\varphi,x)}=M$. 
	
    \begin{rem} \label{remarktransflows}
		Given a flow $\varphi^t:M\to M$, if there is $T\in \R^+$ such that $\varphi^T$ is transitive as a diffeomorphism, then the flow $\varphi^t$ is transitive. This is clear since:
	$$ M = \overline{\bigcup_{n\in \mathbb{N}}(\varphi^T)^n(x)}=\overline{\bigcup_{n\in \mathbb{N}}\varphi^{nT}(x)}\subseteq \overline{\bigcup_{t\in \R^+}\varphi^t(x)} \subseteq M.
	$$
    The opposite direction is not true; a simple counterexample is the linear irrational flow in the torus, where every orbit is dense, but the orbits of the time 1 map leave invariant some transversal sections.
	\end{rem}
  
Finally, we say that a Riemannian metric is transitive if its corresponding geodesic flow is transitive.

\subsection{A criterion for openness} \label{scopeness}
	
In this subsection, we are going to state a result that we are going to apply in the main theorem. Roughly speaking, it says that given a continuous map between topological spaces of the same dimension, and such that the fibers (preimages of points) of the map are small enough, then the image contains an open set. The version we are going to use comes from \cite{LZ}, which is a quantitative version of a result in \cite{BK}. We begin with a few definitions.

\begin{df}
		Suppose $f:X\to Y$ is a continuous map between metric spaces. We say that $y\in Y$ is a \textit{stable value} if there is $\epsilon>0$ such that if $d_{C^0}(f,g)<\epsilon$ then $y\in Im(g)$.
\end{df}
	
\begin{rem} \label{remcontieneabierto}
	Let $Y=\mathbb{R}^n$ and suppose that $f:X \to \mathbb{R}^n$ has a stable value $y$, then $Im(f)$ contains an open set. To see this, take $\epsilon>0$ from the definition of stable value, and a vector $v\in \mathbb{R}^d$ with $\norm{v}<\epsilon$. The map $g:X\to \mathbb{R}^d$ defined by $g(x)=f(x)-v$ satisfies $d_{C^0}(f,g)\leq \norm{v}<\epsilon$. Since $y$ is a stable value, there is $x\in X$ such that $g(x)=y$, and this is equivalent to $f(x)=y+v$. Since $v$ was arbitrary we get $B_{\mathbb{R}^n}(y,\epsilon)\subset Im(f)$.
\end{rem}
	
\begin{df} \label{dflight}
	Given a continuous function $f:X\to Y$ and $\rho>0$ we say that $f$ is $\rho$-light if for every $y \in Y$ the connected components of $f^{-1}(y)$ have diameter smaller than $\rho$.
\end{df}
	
\begin{prop}[Theorem F in \cite{LZ}] \label{propLZadapted} Given $d\in \mathbb{N}$ and $r>0$ there is $\rho=\rho(d,r)>0$ such that every $\rho$-light map $f:[-r,r]^d\to \mathbb{R}^d$ has a stable value. 
\end{prop}
	
The version stated in \cite{LZ} is for maps $f:[0,1]^d \to \mathbb{R}^d$, but the proof can be adapted to maps $f:[-r,r]^d\to \mathbb{R}^d$ for a fixed $r>0$. Combining this proposition and Remark \ref{remcontieneabierto}, we have the following corollary.
	
\begin{cor} \label{corcontieneabierto}
	Fix $d\in \mathbb{N}$ and $r>0$, and take the corresponding $\rho=\rho(d,r)>0$ from Proposition \ref{propLZadapted}. Then there exists $\varepsilon_0>0$ such that the image of every $\rho$-light map $f:[-r,r]^d\to \mathbb{R}^d$ contains an open disk of diameter at least $\varepsilon_0$. 
\end{cor}

Since differentiable manifolds are locally Euclidean spaces, we get an analogous result for maps defined on manifolds. Let us recall that given $M$ a differentiable manifold of dimension $m$, we say that $N$ is a parametrized submanifold of dimension $d$ if there is $U$ an open subset of $\R^d$ and a $C^1$ immersion $\a:U \to M$ such that $\a (U)=N$. 

\begin{prop} \label{contopenmanifolds}
Let $N \subset M$ be a parametrized submanifold of dimension $d$, such that its inner radius is larger than $r>0$. Then, there is $\rho=\rho(d,r)$ and $\varepsilon_0>0$ such that the image of every $\rho$-light map $f:N \to \R^d$ contains an open disk of diameter at least $\varepsilon_0$.  
\end{prop}

\subsection{Proof of the transitivity criterion} \label{ssrtcf}

Before getting into the proof of Theorem \ref{criterion of robust transitivity} itself, let us prove a few lemmas that we are going to use. The first one shows that we can mix open sets inside the stable and unstable sets by using the transitivity; this is going to be applied to the initial geodesic flow. The second one shows that we can recover the intersection if the semiconjugacy was applied to open sets. 

\begin{rem}
    For the next lemma, we are also going to denote by $W^s$ and $W^u$ the stable and unstable sets inside the local product structure. One should have in mind that these sets do coincide in general, but for the sake of clearness of the exposition and since the argument is local, we want to avoid taking local projections along the flow lines several times and exchanging from one notation to the other.
\end{rem}

\begin{lem}
\label{mix open in right angle}
Let \((M,g_0)\) be a Riemannian manifold without conjugate points such that its geodesic flow \(\varphi^t\) on \(T^1M\) is expansive. Given any two open subsets \(V \subset W^s(\theta)\) and \(U \subset W^u(\eta)\), there exists \(T > 0\) such that 
\[
\varphi^T(U) \cap V \neq \emptyset.
\]
\end{lem}

\begin{proof}
Fix $\epsilon>0$ small enough that we can choose small open disks $D^s \subset V \subset W^s(\theta)$ and $D^u \subset U \subset W^u(\eta)$ such that $d(D^u, \partial U)\ge 3\epsilon$ and $d(D^s, \partial V)\ge 3\epsilon$. These guarantee not only that $\overline{D^s}\subset V$ and $\overline{D^u}\subset U$, but also that $\varepsilon$-balls centered in points of these closures are still subsets of $V$ and $U$.

Since $D^s$ is open in $W^s(\theta)$, there exists a decreasing sequence of open sets $\mathcal{V}_n$ inside the local product structure such that $\bigcap_n \mathcal{V}_n = D^s$ and $\mathcal{V}_n \cap W^s(\theta) = D^s$ for all $n$. 
Let $\delta_n>0$ with $\delta_n\to 0$. Then the set $\widetilde{\mathcal{V}}_n :=\displaystyle\bigcup_{|t|<\delta_n} \varphi^t(\mathcal{V}_n)$ is open in $T^1M$ and $\bigcap_n \widetilde{\mathcal{V}}_n = D^s$.

Analogously, since $D^u$ is open in $W^u(\eta)$, there exists a small open set $\mathcal{U}$ in the local product structure whose stable diameter is smaller than $\frac{\epsilon}{2}$. For $\delta>0$ small, the set $\widetilde{\mathcal{U}} :=\displaystyle\bigcup_{|t|<\delta} \varphi^t(\mathcal{U})$ is open in $T^1M$.

It is worth noting that if there exists $\widetilde{T}>0$ large such that $\varphi^{\widetilde{T}}(\widetilde{\mathcal{U}})\cap \widetilde{\mathcal{V}}_n\neq \emptyset$, then there exists $T>0$ such that $\varphi^{T}(\mathcal{U})\cap \mathcal{V}_n\neq \emptyset$. Therefore, by topological transitivity (Theorem \ref{teoLPSruggiero}), we can assume that for each $n$ there exist $T_n > 0$ and $x_n \in \mathcal{U}$ such that $\varphi^{T_n}(x_n) \in \mathcal{V}_n$. 

Since $\mathcal{V}_n \searrow D^s$, we can assume that $\varphi^{T_n}(x_n)\to z\in \overline{D^s}$.

If $\{T_n\}$ were bounded, then, passing to subsequences if necessary, we can assume $T_n \to T < \infty$ and $x_n \to x \in \mathcal{U}$, and by continuity $\varphi^T(x) = z$. However, by the definition of $\mathcal{U}$ there exists $y\in D^u$ such that $y\in W^s_{\epsilon/2}(x)$. Thus, $\varphi^T(y)\in \varphi^T(W^s_{\epsilon/2}(x))\subset W^s_{\epsilon/2}(\varphi^T(x))=W^s_{\epsilon/2}(z)\subset V$, and the desired intersection is obtained.

Now we deal with the case when the sequence $(T_n)_n$ is not bounded. With no loss of generality (we could consider a subsequence), let us assume that $T_n \to \infty$. From the structure of $\mathcal{U}$, for each $n$ we can find $y_n\in W^s_{\varepsilon/2}(x_n)\cap D^u$. By the triangular inequality, we get 
$$d(\varphi^{T_{n}}(y_n),z)\leq d(\varphi^{T_{n}}(y_n), \varphi^{T_{n}}(x_n))+d(\varphi^{T_{n}}(x_n), z).$$
Since $T_n\rightarrow \infty$, it follows that $\varphi^{T_{n}}(y_n)\to z\in \overline{D^s}$.

Remark \ref{local unstable} implies that for $n$ large enough, $W^u_{\epsilon/2}(\varphi^{T_{n}}(y_n))$ is close to $W^u_{\epsilon/2}(z)$. Consequently, $W^u_{\epsilon/2}(\varphi^{T_{n}}(y_n))\cap D^s\neq \emptyset$ for $n$ sufficiently large. Thus, choose $z_n \in D^s$ such that $z_n \in W^u_{\epsilon/2}(\varphi^{T_n}(y_n)) \cap D^s$.
Passing to a subsequence once again, let us assume that $y_n \to y \in \overline{D^u}$. 

The initial choice of $D^u$ implies  that $D^u_{\epsilon}(y):=B_\varepsilon(y)\cap W^u(y) \subset U$. It is easy to see from the definition of unstable sets that 
\[
\varphi^{-T_n}\bigl( W^u_{\epsilon/2}(\varphi^{T_n}(y_n)) \bigr) \subset B_{\epsilon/2}(y_n) \cap W^u(y_n).
\]
Since $y_n \to y$, and $\epsilon/2$ is small, for sufficiently large $n$ we have $B_{\epsilon/2}(y_n) \cap W^u(y_n) \subset D^u_{\epsilon}(y) \subset U$.

By defining $q_n = \varphi^{-T_n}(z_n)$, we get that $q_n \in \varphi^{-T_n}\bigl( W^u_{\epsilon/2}(\varphi^{T_n}(y_n)) \bigr) \subset U$ for $n$ large enough, and $\varphi^{T_n}(q_n) = z_n \in D^s$ for all $n$. This yields $q_n \in U$ with $\varphi^{T_n}(q_n) \in D^s\subset V$ for $n$ large enough, as desired.
\end{proof}
\begin{rem}\label{Remark - Intersection}
In the construction of Lemma 4.10, we can choose the disks $D^u\subset U$ and $D^s\subset V$ such that
\[
d(D^u,\partial U)\ge 3\varepsilon\quad\text{and}\quad d(D^s,\partial V)\ge 3\varepsilon,
\]
with $\varepsilon = \frac{1}{6}\min\{\operatorname{diam}(U),\operatorname{diam}(V)\}$.  
The proof then yields $T>0$ and a point $p\in \varphi^T(U)\cap V$ satisfying
\[
\operatorname{dist}\bigl(p,\partial\varphi^T(U)\bigr)\ge 3\varepsilon = \frac{1}{2}\min\{\operatorname{diam}(U),\operatorname{diam}(V)\}
\]
and
\[
\operatorname{dist}\bigl(p,\partial V\bigr)\ge 3\varepsilon = \frac{1}{2}\min\{\operatorname{diam}(U),\operatorname{diam}(V)\}.
\]
\end{rem}

Next, we present a technical Lemma that simplifies the exposition of the proof for Theorem \ref{criterion of robust transitivity}.
\begin{lem}\label{remove h}
    Let $h_{\psi}$ be a semi-conjugacy given by Theorem \ref{TeoRugtopstab}. If $h_{\psi}$ is $\delta$-close to the identity, then it holds: let two open sets $U$ and $V$ be such that $h_\psi(U)\cap W^u(\theta)$ and $h_\psi(V)\cap W^s(\eta)$ contain open disks, $\mathcal{D}^u$ and $\mathcal{D}^s$, with diameter at least $8\delta$, then there exists $T \in \mathbb{R}$ such that 
    \[
    \psi^{T}(U) \cap V \neq \emptyset.
    \]
    \end{lem}
    \begin{proof}
        By  the previous Lemma, there exists $t_1>0$ such that 
        \[\varphi^{t_1}(\mathcal{D}^u)\cap \mathcal{D}^ s\neq \emptyset.
        \]
        Moreover, this intersection can be done between much smaller disks, that is, if $z$ lies in this intersection, then we can assume that its distance to the boundary of $\varphi^{t_1}(\mathcal{D}^u)$ and $\mathcal{D}^s$ is at least $2\delta$. 

        If we write $h_\psi(h_\psi^{-1}(\mathcal{D}^u)\cap U)\subset h_\psi(U)$ and $h_\psi(h_\psi^{-1}(\mathcal{D}^s)\cap V)\subset h_\psi(V)$, the intersection above can be written as

        \[\varphi^{t_1}(h_\psi(h_\psi^{-1}(\mathcal{D}^u))\cap U)\cap h_\psi(h_\psi^{-1}(\mathcal{D}^s)\cap V)\neq \emptyset.
        \]
        
        By the semi-conjugacy property, we have that there exists $t_2$ such that $\varphi^{t_1}(h_\psi(h_\psi^{-1}(\mathcal{D}^u))\cap U)=h_\psi\circ\psi^{t_2}(h_\psi^{-1}(\mathcal{D}^u)\cap U)$. Hence,
        \[
        h_\psi\circ\psi^{t_2}(h_\psi^{-1}(\mathcal{D}^u)\cap U)\cap h_\psi(h_\psi^{-1}(\mathcal{D}^s)\cap V)\neq \emptyset.
        \]
    Now, we want to remove $h_\psi$ from the intersection. But since $h_\psi$ is $\delta$-close to the identity, the preimage of each of the disks is displaced by at most $\delta$ either in the stable direction, in the unstable direction, or in the flow direction. If the displacement happens to be in the stable or unstable direction, the disks continue to intersect since the intersection happens at a distance of at least $2\delta$  from their boundaries. So $T=t_2$ works. Otherwise, the displacement occurs in the flow direction, and we can use a local projection by the flow lines to find $t_3\in \R$ such that 
    \[
        \psi^{t_3+t_2}(U)\cap V\supset \psi^{t_3+t_2}(h_\psi^{-1}(\mathcal{D}^u)\cap U)\cap (h_\psi^{-1}(\mathcal{D}^s)\cap V)\neq \emptyset.
        \]
        In any case, we find $T\in \R$ such that the Lemma holds.
    \end{proof}

We are now ready to prove Theorem \ref{criterion of robust transitivity}. Let us restate it in detail:

\begin{thm} \label{teoTransitividadeRobusta}
	Let $g_1$ be a $C^2$ Riemannian metric on a compact differentiable manifold $M$ of dimension $n$ with no conjugate points and let $\varphi^t:T^1M\to T^1M$ be its geodesic flow. Suppose that $\varphi^t$ is expansive with stable sets $W^s_{\varphi}$ and unstable sets $W^u_{\varphi}$. Suppose that in addition $\varphi^t$ is partially hyperbolic with a splitting $T(T^1M)=E^{ss}_{\varphi}\oplus E^c_{\varphi} \oplus \langle X \rangle \oplus E^{uu}_{\varphi}$, and it has the SH-Saddle property of index $(d,d)$ where $d=n-1-\textnormal{dim}E^{ss}_{\varphi}=n-1-\textnormal{dim}E^{uu}_{\varphi}$. 
	
	Then there is $\cU$ a $C^2$ neighborhood of $g_1$ (or $\mathcal{V}$ a $C^1$ neighborhood of $\varphi^t$) such that if $g\in \cU$, then the geodesic flow of $g$ is transitive.
	
	 (\emph{i.e.} $\varphi^t$ is $C^1$ robustly transitive, or $g_1$ is $C^2$ robustly transitive among metrics).  
\end{thm}

\begin{proof}	
Let $\varphi^t:T^1M\to T^1M$ be the geodesic flow of the metric $g_1$. By hypothesis, we know there is $T\in \R$ such that $f_1=\varphi^T: T^1M\to T^1M$ is a partially hyperbolic diffeomorphism with a splitting of the form $T(T^1M)=E^{ss}_{f_1}\oplus E^c_{f_1} \oplus E^{uu}_{f_1}$ where $E^{ss}_{f_1}=E^{ss}_{\varphi}$, $E^{uu}_{f_1}=E^{uu}_{\varphi}$ and $E^c_{f_1}=E^c_{\varphi}\oplus \langle X \rangle$. Let us denote by $\cC^{ss}$ and $\cC^{uu}$ the strong stable and strong unstable cones of $f_1=\varphi^T$. Since partial hyperbolicity is $C^1$-open among diffeomorphisms, we know there is a $C^2$ neighbourhood $\cU_0$ of $g_1$ (or a $C^1$ neighbourhood $\mathcal{V}_0$ of $\varphi^t$) such that, if $g\in \cU_0$ and $f:=\psi^T$ where $\psi^t$ is the geodesic flow of $g$, then the same family of strong stable and strong unstable cones holds for $f=\psi^T$.

In addition $f_1$ has the SH-Saddle property of index $(d,d)$, where $d=n-1-k$, and since SH-Saddle property is also $C^2$-open among partially hyperbolic Riemannian metrics by Proposition \ref{SHopenformetrics} (see also Proposition \ref{shfflowsisopen} and Theorem \ref{shisopen}), we know there are constants $\lambda>1$, $L>0$, $\delta_1>0$ and a $C^2$ neighbourhood $\cU_1$ of $g_1$ (or a $C^1$ neighbourhood $\mathcal{V}_1$ of $\varphi^t$) such that, if $g\in \cU_1$ and $f:=\psi^T$ where $\psi^t$ is the geodesic flow of $g$, then: 
	\begin{eqnarray*}
		H^-_{\lambda,d}(f)\cap \cW^{ss}_f(\theta,L)\neq \emptyset \ \ \text{for every} \ \theta \in T^1M, \\
		H^+_{\lambda,d}(f)\cap \cW^{uu}_f(\theta,L)\neq \emptyset \ \ \text{for every} \ \theta \in T^1M. 
	\end{eqnarray*}	
	Moreover by Corollary \ref{corotamanodisco}, for such $f$ and $\theta^u \in H^+_{\lambda,d}(f)$ and $\cD^u$ a center disk of dimension $d$ tangent to $\cC^u(\theta^u)$, there is $N>0$ such that $f^n(\cD^u)$ contains a disk of diameter bigger than $2\delta_1$ for every $n\geq N$. The same happens with the stable manifold: for every $\theta^s \in H^-_{\lambda,d}(f)$ and $\cD^s$ a center disk of dimension $d$ tangent to $\cC^s(\theta^s)$, there is $N>0$ such that $f^{-n}(\cD^s)$ contains a disk of diameter bigger than $2\delta_1$ for every $n\geq N$. Recall that $\cC^s$ and $\cC^u$ are the center cones invariant for the past and the future, respectively, given by the SH-Saddle property. 
	
	
	Now let us define the following constant:
 $$	\rho = \rho(n-1,\delta_1)$$ where $\rho(\cdot \ ,\cdot)$ is given by Corollary  \ref{corcontieneabierto} (see also Proposition \ref{contopenmanifolds}).
	
	According to Theorem \ref{TeoRugtopstab} since $\varphi^t$ is expansive, there is a $C^0$ neighbourhood $\mathcal{V}_2$ of $\varphi^t$ (or a $C^1$ neighborhood $\cU_2$ of $g_1$) such that, for every flow $\psi^t\in \mathcal{V}_2$ there are continuous and surjective functions $h_{\psi}:T^1M \to T^1M$ and $r_{\psi}:T^1M \times \R \to \R$ with $r_{\psi}(\theta,0)=0$ such that: 
		\begin{equation} \label{eqsemiconj} h_{\psi} \circ \psi^t(\theta)=\varphi^{r_{\psi}(\theta,t)}\circ h_{\psi}(\theta) \ \ \text{for every} \ \ t\in \R, \ \theta\in T^1M.
		\end{equation}
		We can take $\mathcal{V}_2$ sufficiently small such that for every flow $\psi^t \in \mathcal{V}_2$ we have $d_{C^0}(h_{\psi},\textit{Id})<\frac{1}{8}\min\{\rho,\varepsilon_0\}$, where $\varepsilon_0$ is the uniform constant given by Proposition \ref{contopenmanifolds}, and moreover (see Lemma \ref{controlfibras}) $\text{diam}(h_{\psi}^{-1}(h_{\psi}(\theta)))<\rho/2$ for every $\theta \in T^1M$ (in other words $h_{\psi}$ is $\rho/2$-light, see Definition \ref{dflight}).
		

Now take $\cU=\cU_0 \cap\cU_1\cap \cU_2$. We claim that every Riemannian metric $g\in \cU$, has a transitive geodesic flow. Take $g\in \cU$ and denote by $\psi^t:T^1M\to T^1M$ its corresponding geodesic flow. By Proposition \ref{propequivtrans} (see also Remark \ref{remarktransflows}) it is enough to prove that for any two open sets  $U_1, U_2 \subset T^1M$ there is $t \in \R^+$ such that $\psi^t(U_1)\cap U_2 \neq \emptyset$.
		
		Take two points $\theta_1\in U_1$ and $\theta_2\in U_2$, and denote by $f:=\psi^T$. Let $n_1 \in \N$ be such that $f^{-n_1}(U_1)\supset \cW^{ss}_f(f^{-n_1}(\theta_1),L)$ and $f^{n_1}(U_2)\supset \cW^{uu}_f(f^{n_1}(\theta_2),L)$. Take $\theta^s \in H^-_{\lambda,d}(f) \cap \cW^{ss}_f(f^{-n_1}(\theta_1),L)$ and $\theta^u \in H^+_{\lambda,d}(f) \cap \cW^{uu}_f(f^{n_1}(\theta_2),L)$ given by $(d,d)$ SH-Saddle property. 
		
		Now take $\cD^s$ a center disk of dimension $d$ tangent to $\cC^s(\theta^s)$ and centered at $\theta^s$, and take $\cD^u$ a center disk of dimension $d$ tangent to $\cC^u(\theta^u)$ centered at $\theta^u$. We can take $\cD^s, \cD^u$ small enough such that $\cD^s \subset f^{-n_1}(U_1)$ and $\cD^u \subset f^{n_1}(U_2)$. 
		
		Now take the following sets:
		$$\Delta^s=\bigcup_{\theta\in \cD^s} \cW^{ss}_f(\theta,l) \ \ \text{and}\ \ \Delta^u=\bigcup_{\theta\in \cD^u}\cW^{uu}_f(\theta,l).$$ 
		We can choose $l>0$ small enough such that $\Delta^s \subset f^{-n_1}(U_1)$ and $\Delta^u \subset f^{n_1}(U_2)$. 
		Notice that $\Delta^s$ and $\Delta^u$ are disks of dimension equal to $n-1$ (recall that $n-1=d+k$). 
		
		Now the idea is to use Corollary \ref{corcontieneabierto} applied to the maps $$P^s_{\psi,\xi^s}:=\Pi^s_{\xi^s} \circ p \circ h_{\psi} \ \ \ \text{and} \ \ \ P^u_{\psi,\xi^u}:=\Pi^u_{\xi^u} \circ p \circ h_{\psi} $$ where $\xi^s=h_{\psi}(\theta^s)$ and $\xi^u=h_{\psi}(\theta^u)$, the maps $\Pi^s_{\theta}:\Sigma_{\theta}\to \Sigma^s_{\theta}$ and $\Pi^u_{\theta}:\Sigma_{\theta}\to \Sigma^u_{\theta}$ are the local projections restricted to the local section $\Sigma_\theta$ given by the local product structure of the flow, and the function $p$ is the local projection from $T^1M$ to $\Sigma_\theta$ given by the flow lines of $\varphi$. 
		

	\begin{claim}
		The functions $P^s_{\psi,\xi^s}\big|_{\Delta^s}:\Delta^s \to \Sigma^s_{\xi^s}$ and $P^u_{\psi,\xi^u}\big|_{\Delta^u}:\Delta^u\to \Sigma^u_{\xi^u}$ are $\rho$-light.
	\end{claim}	

    \begin{proof}[Proof of the claim]
		Let us see the case $P^u_{\psi,\xi^u}\big|_{\Delta^u}:\Delta^u \to \Sigma^u_{\xi^u}$ since the other one is symmetric. 
		Notice that the disk $\Delta^u$ is tangent to the cones $\mathcal{C}^{uu} \times \mathcal{C}^u$ where $\mathcal{C}^{uu}$ is the strong unstable cones given by the partial hyperbolicity of $\varphi^T$. Recall that the same family of cones holds for $f=\psi^T$, since the flows $\psi^t$ and $\varphi^t$ are $C^1$ close enough. Moreover, this family of cones is topologically transverse to the center stable sets of $\varphi$. In particular, we have that
		$$ \# \left\{ \Delta^u \cap W^{cs}_{\varphi, loc}(\eta) \right\} = 1 \ \ \text{for every} \ \eta \in \Delta^u. $$ 
		Since $h_{\psi}$ is $\rho/2$-light it is clear that
		$$ \text{diam} \left\{ h_{\psi}(\Delta^u) \cap W^{cs}_{\varphi, loc}(\eta)\right\} < \rho \ \ \text{for every} \  \eta \in \Delta^u$$
		and this is equivalent to the the $\rho$-lightness of the function $P^u_{\psi,\xi^u}\big|_{\Delta^u}:\Delta^u\to \Sigma^u_{\xi^u}$.

\end{proof}
In the same way, for every $n\geq 0$ we consider the maps
$$P^s_{\psi,\xi^s_n}\big|_{\Delta^s_n}:\Delta^s_n \to \Sigma^s_{\xi^s_n} \ \ \text{and} \ \  P^u_{\psi,\xi^u_n}\big|_{\Delta^u_n}:\Delta^u_n\to \Sigma^u_{\xi^u_n}$$
where:
\begin{itemize}
	\item $\Delta^s_n=f^{-n}(\Delta^s)$ and $\Delta^u_n=f^n(\Delta^u)$.
	\item $\xi^s_n=h_{\psi}(f^{-n}(\theta^s))$ and 	$\xi^u_n=h_{\psi}(f^n(\theta^u))$.
	\item $P^s_{\psi,\xi^s_n}:=\Pi^s_{\xi^s_n} \circ p \circ h_{\psi}$ and $P^u_{\psi,\xi^u_n}:=\Pi^u_{\xi^u_n} \circ p \circ h_{\psi}$.
\end{itemize}
The functions $\Pi^s_{\theta}:\Sigma_{\theta}\to \Sigma^s_{\theta}$, $\Pi^u_{\theta}:\Sigma_{\theta}\to \Sigma^u_{\theta}$ and $p$ are the same as above. Notice that basically $P^s_{\psi,\xi^s_n} \simeq P^s_{\psi,\xi^s} \circ f^{-n}$ and $P^u_{\psi,\xi^s_n} \simeq P^u_{\psi,\xi^s} \circ f^{n}$. Since the family of disks $\Delta^s_n$ is also tangent to the cones $\cC^{ss}\times \cC^s$ and the family of disks $\Delta^u_n$ is tangent to the cones $\cC^{uu}\times \cC^u$, in the same way as above, we get the following claim. 


\begin{claim}
	The functions $P^s_{\psi,\xi^s_n}\big|_{\Delta^s_n}:\Delta^s_n \to \Sigma^s_{\xi^s_n}$ and $P^u_{\psi,\xi^u_n}\big|_{\Delta^u_n}:\Delta^u_n\to \Sigma^u_{\xi^u_n}$ are $\rho$-light, for every $n\geq 0$.  
\end{claim} 

 

Now we continue with the proof of the theorem. By Corollary \ref{corotamanodisco} there is $n_2\in \N$ such that for every $n\geq n_2$, $f^{-n}(\cD^s)$ contains a disk of size bigger than $2\delta_1$ centered at $f^{-n}(\theta^s)$. Therefore, $\Delta^s_n=f^{-n}(\Delta^s)$ contains a disk of dimension $n-1$, of size bigger than $2\delta_1$. In the same way, for every $n\geq n_2$, $f^n(\cD^u)$ contains a disk of size bigger than $2\delta_1$ centered at $f^n(\theta^u)$, and therefore $\Delta^u_n=f^{n}(\Delta^u)$ contains a disk of dimension $n-1$, of size bigger than $2\delta_1$.

To finish the proof just notice that for every $n\geq n_2$, the function $P^u_{\psi,\xi^u_n}\big|_{\Delta^u_n}:\Delta^u_n\to \Sigma^u_{\xi^u_n}$ is in the hypotheses of Corollary \ref{corcontieneabierto} since $D^u_n$ is a disk of size bigger than $2\delta_1$ centered at $f^n(\theta^u)$ and by the previous claim $P^u_{\psi,\xi^u_n}$ is $\rho$-light. Then $P^u_{\psi,\xi^u_n}(\Delta^u_n) \subset \Sigma^u_{\xi^u_n}$ contains an open set, for every $n\geq n_2$. Moreover, by Proposition \ref{contopenmanifolds}, it contains an open disk of diameter at least $\varepsilon_0$.

In the same way, $P^s_{\psi,\xi^s_n}\big|_{\Delta^s_n}:\Delta^s_n \to \Sigma^s_{\xi^s_n}$ is in the hypotheses of Corollary \ref{corcontieneabierto}, and therefore $P^s_{\psi,\xi^u_n}(\Delta^s_n) \subset \Sigma^s_{\xi^u_n}$ contains an open set, for every $n\geq n_2$. Again, it contains an open disk of diameter at least $\varepsilon_0$.


Since the flow $\varphi^t$ is transitive (by Theorem \ref{teoLPSruggiero}) and the subsets $P^u_{\psi,\xi^u_n}(\Delta^u_n)$ and $P^s_{\psi,\xi^s_n}(\Delta^s_n)$ contain topological disks of complementary dimensions, a direct application of Lemma \ref{mix open in right angle} shows that there is $t_1>0$ such that for any $n \geq n_2$ we have 
$$\varphi^{t_1}(h_{\psi} (\Delta^u_n))\cap h_{\psi}(\Delta^s_n) \neq \emptyset.$$
Recall that $\Delta^u_n=f^n(\Delta^u) \subseteq f^{n+n_1}(U_2)$ and $\Delta^s_n=f^{-n}(\Delta^s) \subseteq f^{-n-n_1}(U_1)$, then this is the same as
$$\varphi^{t_1}(h_{\psi} (f^{n+n_1}(U_2)))\cap h_{\psi}(f^{-n-n_1}(U_1)) \neq \emptyset.$$
By our choice of $h_\psi$ at $C^0$-distance at most $\frac{\varepsilon_0}{8}$ from the identity, we can argue as in the proof of Lemma \ref{remove h} to obtain $t_2 \in \R$ such that for every $n\geq n_2$ we have
$$\psi^{t_2}(f^{n+n_1}(U_2))\cap f^{-n-n_1}(U_1) \neq \emptyset.$$ 

Since $f=\psi^T$, this implies that
$$\psi^{t_2+T(n+n_1)}(U_2) \cap \psi^{-T(n+n_1)}(U_1) \neq \emptyset$$ which is equivalent to
$$\psi^{t_2+2T(n+n_1)}(U_2)\cap U_1 \neq \emptyset$$ for every $n\geq n_2$.  
Since the choice of $U_1$ and $U_2$ was arbitrary, this proves that $\psi^t$ is topologically transitive. 
\end{proof}

\begin{rem}
    Let us bring attention to the elements that are crucial to our proof. Indeed, the proof is purely topological and dynamical and does not rely heavily on the nature of the geodesic flow. The first one is partial hyperbolicity and the SH-Saddle property for flows. Proposition \ref{shfflowsisopen} gives us that these are $C^1$-open properties inside the class of flows. The second element is expansivity with local product structure for the stable and unstable sets. These imply topological transitivity, density of periodic points, and topological stability. The same proof we have presented for Theorem \ref{teoTransitividadeRobusta} works for the general class of flows, and we will state it for future use.
\end{rem}
\begin{thm}
    Let $\varphi^t: M\rightarrow M$ be a partially hyperbolic flow of class $C^1$ on a compact and connected differentiable manifold. Assume that $\varphi^t$ satisfies the following conditions:
    \begin{enumerate}
        \item $\varphi^t$ is expansive with local product structure between the stable set $W^s$ and the unstable set $W^u$.
        \item $\varphi^t$ has SH-Saddle property of index $(d_1,d_2)$ where $d_1=\textnormal{dim}W^s - \textnormal{dim}E^{ss}$ and $d_2=\textnormal{dim}W^u - \textnormal{dim}E^{uu}$.
    \end{enumerate}
    Then, $\varphi^t$ is $C^1$-robustly transitive.
\end{thm}

\section{Proof of Theorem \ref{THM A}}
\label{proof of main theorem}
In this section, we prove Theorem \ref{THM A}. First, in Subsection \ref{ssExSH}, we show that there are examples of Riemannian metrics that present the SH-Saddle property. Then, in Subsection \ref{ssexpandts}, we conclude the proof of the theorem by applying Theorem \ref{criterion of robust transitivity} to those examples.

\subsection{Existence of the SH-Saddle property} \label{ssExSH}
In this subsection, we are going to present examples satisfying the conditions of Theorem \ref{criterion of robust transitivity} in order to obtain robustly transitive geodesic flows. In fact, the examples are those constructed in \cite{dJPR}, and we quickly recover the construction here for the sake of completeness. The following results are obtained in \cite{dJPR}:

\begin{thm}[Theorem B in \cite{dJPR}]
\label{ergodic partially hyperbolic geodesic flow}
    Let $(M,g)$ be a compact Kähler manifold of holomorphic curvature $-1$ or a compact locally symmetric quaternionic Kähler manifold of negative curvature. Then there exists a $C^2-$deformation $\tilde g$ of the metric $g$ with the following properties:
    \begin{enumerate}
        \item The geodesic flow of $\tilde g$ is partially hyperbolic.
        \item There exists a closed geodesic $\gamma$ with a parallel Jacobi field along it.
        \item The sectional curvatures $\tilde K$ are all negative outside $\gamma$.
    \end{enumerate}
\end{thm}

\begin{cor}[Corollary B.1 in \cite{dJPR}]
\label{new geodesic flow has all the good properties}
    There exists a Riemannian manifold $(M,\tilde g)$ with no conjugate points such that its geodesic flow is partially hyperbolic, non-Anosov, ergodic for Liouville, mixing, expansive, and has a unique measure of maximal entropy.
\end{cor}

From the previous corollary, in order to apply Theorem \ref{criterion of robust transitivity} to the Riemannian metric $\tilde g$ above, the only hypothesis that we are left to check is the SH-saddle property. Before showing this property, let us first recover the construction of the example by mentioning the main steps and ideas:
\begin{enumerate}
    \item Start with a locally symmetric Riemannian manifold of non-constant negative curvature $(M,g)$. The geodesic flow for such metrics displays an Anosov splitting of the form $$T_{\theta}(T^1M)=E^{ss}\oplus E^{ws} \oplus \R G(\theta) \oplus E^{wu}\oplus E^{uu}$$ where $\R G(\theta)$ is the flow direction. We can see this dominated splitting as a partially hyperbolic splitting with center bundle $E^c= E^{ws} \oplus E^{wu}$.
    \item Fix a simple closed geodesic $\gamma$ of period $T$ for the metric $g$. We can always find such a geodesic in a compact manifold \cite{Kl2}.
    
    \item We make a deformation of $g$ inside a small tubular neighbourhood $B(\gamma,\varepsilon)$ to obtain a new metric $\tilde g$. Since $\tilde g|_{B(\gamma,\varepsilon)^c}=g|_{B(\gamma,\varepsilon)^c}$ the dynamics of the new geodesic flow remains hyperbolic outside $B(\gamma, \varepsilon)$. 
     
    \item We change the metric inside $B(\gamma, \varepsilon)$ in a way that at least one vector in the center bundle is not contracted nor expanded by the new geodesic flow, which implies that it is not globally hyperbolic (Anosov). 

    \item By taking $\varepsilon$ sufficiently small, the geodesic flow of the new metric $\tilde g$ is still partially hyperbolic.

    \item We show that there exists $\varepsilon_0>0$ such that the construction works for all $0<\varepsilon\leq \varepsilon_0$. This will be important in our proof here since it gives us some freedom to consider the construction for smaller $\varepsilon>0$ if necessary.
\end{enumerate}

Now take the Riemannian metric $\tilde g$ as above and denote by $\varphi^t:T^1M \to T^1M$ the geodesic flow of $\tilde g$. Recall that in order to prove that $\tilde g$ has the SH-Saddle property, we have to find $T>0$ such that the map $f:=\varphi^T$ has the SH-Saddle property as a diffeomorphism. Now we fix $T>0$ and consider $f=\varphi^T$. Suppose that there is a point $\theta^u \in T^1M$ such that its forward orbit never enters the tubular neighborhood $B(\gamma,\varepsilon)$ in the future, i.e. $f^n(\theta^u)=\varphi^{Tn}(\theta)\notin B(\gamma,\varepsilon)$ for every $n\geq0$. Then, since the metric $\tilde g$ is equal to $g$ outside $B(\gamma,\varepsilon)$, and since the geodesic flow of $g$ is hyperbolic, the forward orbit of $\theta^u$ by $f$ behaves hyperbolically.  

Recall the SH-Saddle property for the unstable manifolds means that every unstable leaf of a given large has a point which behaves hyerbolic for the future, so in order to get the SH-Saddle property we only have to check that in every strong unstable leaf of a given large, say $\mathcal{W}^{uu}_f(\theta,L)$, there is a point $\theta^u \in \mathcal{W}^{uu}_f(\theta,L)$ that never enters in the tubular neighborhood for the future. This will be enough to prove that the strong unstable foliation has the SH-Saddle property. Analogously, the same will hold for the strong stable manifold but iterating for the past.

The idea to find points in the unstable manifold that never enter some given transverse region is classical and appears in \cite{Mane2} and also in \cite{HHU} in a more general setting (named as ``keepaway Lemma"). The same proof also appears in \cite{Pi2}.

\begin{lem} \label{ExampleisSH}
	The geodesic flow of $\tilde g$ has the SH-Saddle property of index $(d,d)$, where $d=n-1-\textnormal{dim}(E^{ss})=n-1-\textnormal{dim}(E^{uu})$.
\end{lem}
\begin{proof}
    Let us denote by $\varphi^t:T^1M \to T^1M$ the geodesic flow of the Riemannian metric $\tilde g$. The idea is to prove that for some $T>0$, the map $f=\varphi^T$ has the $(d,d)$ SH-saddle property (recall Definition \ref{defshf}). We are going to show that, for a $T>0$, the strong unstable manifold $\cW^{uu}_f$ has the SH-saddle property of index $d$. The stable part is completely analogous.
	
	Recall that the geodesic flow of the original metric $g$ has a dominated splitting of the form:
	\begin{equation} \label{splittinghyp}
	T_{\theta}(T^1M)=E^{ss}\oplus E^{ws} \oplus  \R G(\theta) \oplus E^{ws}\oplus E^{uu}
    \end{equation}
    where $\R G(\theta)$ is the flow direction, 
	and the new metric $\wt{g}$ is equal to $g$ outside the tubular neighborhood $B(\gamma,\e)$ by definition. 

Now the key observation is that the strong bundles $E^{ss}$ and $E^{uu}$ are uniformly transverse to $TB(\gamma, \epsilon)=\left\{ (p,v) \in TM: p \in B(\gamma, \epsilon)\right\}$. In particular, every unstable leaf of size $L$, $\cW^{uu}(\theta,L)$, with $L$ sufficiently large compared to $\e$, crosses $TB(\gamma,\epsilon)$ a finite number of times. Moreover there is $\delta>0$ and a point $\theta_1 \in \cW^{uu}(\theta,L)$ and $\delta>0$ such that: 
\begin{equation*}  \cW^{uu}(\theta_1,\delta) \subset \cW^{uu}(\theta,L)\setminus TB(\gamma,\epsilon).
\end{equation*}
	Now we can take $T>0$ large enough, such that 
	\begin{equation} \label{eqSH}\delta(\lambda_{uu})^T> L.
	\end{equation}	
	For this $T$, we consider the partially hyperbolic diffeomorphism $f=\varphi^T$. We claim that the unstable foliation of $f$ has the SH-saddle property of index $d$. We proceed by induction as follows:

	Take $\theta \in T^1M$ and consider $\cW^{uu}_f(\theta,L)$. As we already mentioned, there is a point $\theta_0$ such that $\cW^{uu}_f(\theta_0,\delta) \subset \cW^{uu}_f(\theta,L)\setminus TB(\gamma,\epsilon)$. We denote by $\cD_{0}=\overline{\cW^{uu}_f(\theta_0,\delta)}$. By our choice of $T$ and $\delta$ in Equation \eqref{eqSH}, we have  $f(\cD_{0})\supseteq\mathcal{W}^{uu}_{f}(f(\theta_0),L)$ and in the same way as above, we can find a disk $\cD_{1}=\overline{\mathcal{W}^{uu}_f(\theta_1,\delta)} \subset f(\cD_0)$ such that $\cD_1\cap TB(\gamma,\e)=\emptyset$. Inductively we get a sequence of unstable disks $\{\cD_j\}_{j\geq0}$ such that  $\cD_j\cap TB(\gamma,\e)=\emptyset$ for every $j\geq 0$ and $f^{-1}(\cD_{j+1}) \subset \cD_{j}$. Finally the point $\theta^u=\bigcap_{j\geq 0}f^{-j}(\cD_j)$ never meets $TB(\gamma,\e)$ in the future, i.e. $f^n(\theta^u) \notin TB(\gamma,\e)$ for every $n \geq 0$. 
    
    Recall that the geodesic flow of $\tilde g$ is equal to the geodesic flow of $g$ outside $TB(\gamma,\e)$, and in this region the tangent bundle has a splitting exactly as in Equation \eqref{splittinghyp}. Therefore, we get the point $\theta^u$ is hyperbolic for the future showing the unstable manifold $\cW^{uu}_f$ has SH-Saddle property of index $d$, where $d=\textnormal{dim}(E^{wu})$. Now just notice that $\textnormal{dim}(E^{wu})=n-1-\textnormal{dim}(E^{uu})$.
	
	In the same way, in every strong stable leaf of large $L$ we can find a point $\theta^s$  such that the past orbit of $\theta^s$ never meets $TB(\gamma,\e)$. Once again since the two geodesic flows coincide outside $TB(\gamma,\e)$, the same argument as above shows that $\cW^{ss}_f$ has SH-Saddle property of index $d=\textnormal{dim}(E^{ws})=n-1-\textnormal{dim}(E^{ss})$.
\end{proof}

\begin{rem}
The proof of Lemma \ref{ExampleisSH} also works for the Carneiro-Pujals examples in \cite{CP}, that is, their examples also present the SH-Saddle property. However, they are not known to be expansive, nor is there any geodesic with conjugate points. These are open questions about their work as well as ergodicity.
\end{rem}


\subsection{Conclusion of the proof} \label{ssConcl}
We now finish the proof of Theorem \ref{THM A} by following the steps we stated in the introduction. 

Let $\tilde g$ be the Riemannian metric given by Corollary \ref{new geodesic flow has all the good properties}, i.e., $\tilde g$ has no conjugate points and its geodesic flow $\varphi^t$ is expansive, partially hyperbolic but non-Anosov. By Lemma \ref{ExampleisSH} the geodesic flow $\varphi^t$ has the SH-saddle property of index $(d,d)$, where $d=n-1-k$ and $k=\text{dim}E^{ss}=\text{dim}E^{uu}$. Then by applying Theorem \ref{criterion of robust transitivity}, we obtain a $C^2$ neighborhood of $\tilde g$, say $\mathcal{U}$, such that every Riemannian metric $ g\in \mathcal{U}$ has a topologically transitive geodesic flow. 

Ruggiero proved in \cite{Rug0} that the $C^2$ interior of the set of Riemannian metrics with no conjugate points coincides with the set of Riemannian metrics with Anosov geodesic flow. Since the geodesic flow of $\tilde g$ is non-Anosov, the metric $\tilde g$ must lie in the $C^2$ boundary of the set of Riemannian metrics with no conjugate points. Therefore, there is a metric $\hat g \in \mathcal{U}$ that has conjugate points. Moreover, by Corollary 1.2 of \cite{Rug0}, we know that the set of Riemannian metrics with conjugate points is $C^2$-open. Then there exists $\mathcal{V}\subset\mathcal{U}$ a $C^2$ neighborhood of $\hat g$ consisting of metrics with conjugate points. 

Therefore, we conclude that every Riemannian metric in $\mathcal{V}$ has conjugate points and a partially hyperbolic geodesic flow that is topologically transitive. 

\section{Further comments and results}
Among other dynamical properties we could find in the boundary of metrics with no conjugate points, it would be very interesting if stable ergodicity could be produced. One could find a criterion for stable ergodicity that applies to the same class of examples, and this would imply the existence of ergodic geodesic flows for metrics with conjugate points in a similar way to the results presented here. The existence of such systems is not known for dimensions higher than $2$. 

Theorem \ref{THM A} gives us an open set of Riemannian metrics with conjugate points and partially hyperbolic topologically transitive geodesic flows. Hence, applying genericity results, we can find examples with further properties. In here, we mention the existence of \textit{bumpy metrics}.

Let $(M,g)$ be a Riemannian metric, with geodesic flow $\varphi^t$, and $\gamma$ be a closed prime geodesic of period $T$ (i.e. $\gamma(0)=\gamma(T)$ and $\gamma'(0)=\gamma'(T))$. The geodesic $\gamma$ is said to be \textit{nondegenerate} if $(D\varphi^T)_{\gamma'(0)}: T_{\gamma'(0)}(T^1M)\rightarrow T_{\gamma'(T)}(T^1M)$ has only one eigenvalue equal to $1$.

\begin{df}
    A Riemannian metric $g$ is called \textit{bumpy} if all its closed geodesics are nondegenerate.
\end{df}

The so-called \textit{bumpy metric Theorem}, whose proof can be found in \cite{An82}, states that a metric can always be perturbed into a bumpy metric:

\begin{thm}[bumpy metric Theorem]
    The bumpy metrics form an $C^r$-everywhere dense subset of the space of Riemannian metrics of class $C^r$, for $2\leq r\leq \infty$.
\end{thm}

Hence, Theorem \ref{THM A} together with the bumpy metric Theorem leads to:
\begin{prop}
    There exists a bumpy metric $g$, with geodesic flow $\varphi^t$ such that:
    \begin{enumerate}
        \item $g$ has conjugate points.
        \item $\varphi^t$ is a partially hyperbolic flow.
        \item $\varphi^t$ is topologically transitive.
    \end{enumerate}
\end{prop}

Regarding the thermodynamical formalism, the examples obtained by Theorem \ref{ergodic partially hyperbolic geodesic flow} have a unique measure of maximal entropy. The geodesic flows that are sufficiently close to these examples are topologically semi-conjugate. If one can produce a perturbation of the examples to get metrics with conjugate points and such that the semi-conjugacy preserves time (i.e., $h\circ\varphi^t=\psi^t \circ h$), then the perturbed one should also have at most one measure of maximal entropy. Therefore, this could also be a property that one could push outside the class of metrics with no conjugate points.

\section*{Acknowledgments} 
The authors are tremendously grateful to Enrique Pujals for several helpful conversations and suggestions. The second author also wishes to thank Rafael Potrie and Mart\'in Sambarino for many discussions about the transitivity of the example in this article.

\textbf{Declaration}. The first named author was funded in part by the Luxembourg National Research Fund
(FNR), Grant reference O24/18936913/RiGA. For the purpose of open access, and in fulfillment
of the obligations arising from the grant agreement, the author has applied a Creative Commons Attribution 4.0 International (CC BY 4.0) license to any Author Accepted Manuscript
version arising from this submission.

\Addresses


\begin{thebibliography}{99}
	
	
\def\bi#1{\bibitem{#1}}

\bi{An} D. V. Anosov, Geodesic flows on closed Riemann manifolds with negative curvature, \textit{Proc. Steklov Inst. Math.}, \textbf{90} (1967), 1-235 {(English translation. Providence, R.I.:Amer. Math. Soc. 1969)}.

\bi{An82} D.V. Anosov, Generic properties of closed geodesics, \textit{Izv. Akad. Nauk SSSR Ser. Mat.} {\bf 46} (1982), no.~4, 675--709, 896.


\bi{BD} C. Bonatti and L.~J. D\'iaz, Persistent nonhyperbolic transitive diffeomorphisms, Ann. of Math. (2) {\bf 143} (1996), no.~2, 357-396.

\bi{BGV} C. Bonatti, N. Gourmelon, and T. Vivier, Perturbations of the derivative along periodic orbits, \textit{Ergodic Theory and Dynamical Systems}, vol. 26, no. 5. (2006). pp. 1307–1337. 

\bi{BK} M. Bonk and B. Kleiner, Rigidity for quasi-M$\ddot{\textnormal{o}}$bius group actions, \textit{J. Diff. Geom.}  \textbf{61} (1), (2002), 81-106.

\bi{BW} K. Burns and H.~N. Weiss, Spheres with positive curvature and nearly dense orbits for the geodesic flow, \textit{Ergodic Theory Dynam. Systems} {\bf 22} (2002), no.~2, 329--348.

\bi{CP} E. Carneiro, E. Pujals. Partially hyperbolic geodesic flows. Ann. Ins. H. Poincar\'e, \textbf{31}(5), (2014), 985-1014. 

\bi{Cont} G. Contreras, Partially hyperbolic geodesic flows are Anosov, \textit{C. R. Math. Acad. Sci. Paris} {\bf 334} (2002), no.~7, 585--590.


\bi{ConP} G. Contreras and G. Paternain. Genericity of geodesic flows with positive topological entropy on $S^2$, \textit{J. Differential Geom.} {\bf 61} (2002), no.~1, 1-49.

\bi{dJPR} Y. de Jesus, L. Pi\~{n}eyr\'ua and S. Roma{\~n}a. Partially hyperbolic geodesic flow via conformal deformation. Preprint available on arXiv: \url{https://arxiv.org/abs/2410.21519} (2024).

\bi{DoC} M. P. do Carmo, Geometria Riemanniana, \textit{Projecto Euclides IMPA} (2015).

\bi{Don1} V.~J. Donnay, Geodesic flow on the two-sphere. I. Positive measure entropy, \textit{Ergodic Theory Dynam. Systems} {\bf 8} (1988), no.~4, 531--553.

\bi{Don2} V.~J. Donnay, Geodesic flow on the two-sphere. II. Ergodicity, {\it Dynamical systems (College Park, MD, 1986--87)}, 112--153, Lecture Notes in Math., 1342, Springer, Berlin.

\bibitem{DonnayPugh2003}
  V. Donnay and C. Pugh.
  Anosov geodesic flows for embedded surfaces,
  \textit{Geometric methods in dynamics (II): Volume in honor of Jacob Palis},
  Ast\'erisque, no. 287, pp. 61--69, Soci\'et\'e Math\'ematique de France, Paris, 2003.

\bi{Eb1} P. Eberlein. When is a geodesic flow of Anosov type? I, \textit{Journal of Differential Geometry}, \textbf{8}, (1973), 437-463. 

\bi{Eb2} P. Eberlein. Structure of manifolds with nonpositive curvature, \textit{Lecture  Notes in Mathematics}, \textbf{1156}, (1985), 86-153. 

\bi{Eb3} P. Eberlein. Geometry of Nonpositively Curved manifolds, \textit{Chicago Lectures in Mathematics} (1996). 

\bi{Eb4} P. Eberlein. Geodesic flows in certain manifolds without conjugate points, \textit{Transactions of the American Mathematical Society}, \textbf{167}, (1972), 437-463. 

\bi{GU} R. D. Gulliver II. On the variety of manifolds without conjugate points, \textit{Trans. Amer. Math. Soc.} {\bf 210} (1975), 185--201.

\bi{HPS} M. Hirsch, C. Pugh and M. Shub. Invariant manifolds, \textit{Springer Lecture Notes in Math.}, \textbf{583} (1977).

\bi{hopf1948} E.Hopf, Closed surfaces without conjugate points,
\textit{Proc. Natl. Acad. Sci. USA} \textbf{34} (1948), 47--51.

\bi{HuWa} W. Hurewicz and H. Wallman. Dimension theory, \textit{Princeton University Press}, (1941), \textbf{4}, Princeton Mathematical Series, 2.

\bi{KW} G. Knieper and H.~N. Weiss, $C^\infty$ genericity of positive topological entropy for geodesic flows on $S^2$, \textit{J. Differential Geom.} {\bf 62}, no.~1, (2002), 127--141.

\bi{Kl2} W. P. A. Klingenberg. Lectures on closed geodesics, Springer-Verlag, 1978.

\bi{Kl}W. P. A. Klingenberg. Riemannian manifolds with geodesic flow of Anosov type, \textit{Ann. of Math.} (2) {\bf 99} (1974), 1-13.

\bi{Mane2} R. Ma\~n\'e, Contributions to the stability conjecture, Topology {\bf 17} (1978), no.~4, 383-396.

\bi{Mane} R. Ma\~n\'e. On a theorem of Klingenberg. \textit{Dynamical Systems and Bifurcation Theory}, M.
Camacho, M. Pacifico and F. Takens eds, Pitman Research Notes in Math, (1987) pages 319–345.

\bi{LRR} A. Lazrag, L. Rifford and R.~O. Ruggiero, Franks' lemma for $C^2$-Ma\~n\'e{} perturbations of Riemannian metrics and applications to persistence, \textit{J. Mod. Dyn.} {\bf 10} (2016), 379--411.

\bi{MR}\'I. Melo and S. Roma\~na, Riemannian manifolds with Anosov geodesic flow do not have conjugate points, \textit{arXiv:2008.12898}, (2022).


\bi{LZ} M. Leguil and Z. Zhang. $C^{r}$-prevalence of stable ergodicity for a class of partially hyperbolic systems, \textit{Journal of Eur. Math. Soc. (JEMS)}, \textbf{24} (9), (2022), 3379-3438. 

\bi{Lew83} J. Lewowicz. Persistence in expansive systems, \textit{Ergod. Theor. $\&$ Dynam. Sys.}, \textbf{3} (4), (1983), 567-578.	

\bi{OS} J. J. O'Sullivan. Manifolds without conjugate points. \textit{Math. Ann.}, \textbf{210}, (1974), 295–311.

\bi{Patphd} M. Paternain. Expansive flows and the fundamental group (PhD Thesis) 1990. 

\bi{Pat93}
M.~Paternain, Expansive geodesic flows on surfaces,
\textit{Ergodic Theory Dynam. Systems} \textbf{13} (1993), no.~1, 153--165.

\bi{Pi1}  L. P. Pi\~{n}eyr\'ua. Contributions to partially hyperbolic systems: coherence, transitivity and ergodicity. PhD Thesis (2022).

\bi{Pi2} L. P. Pi\~{n}eyr\'ua. Some hyperbolicity revisited and robust transitivity, \textit{Ergod. Theor. $\&$ Dynam. Sys.}, \textbf{45} (12), (2025), 3800-3831.

\bi{PS} E. Pujals and M. Sambarino. A sufficient condition for robustly minimal foliations.\textit{ Ergod. Th. $\&$ Dynam.
Sys.} \textbf{26}(1) (2006), 281–289.

\bi{HHU} F. Rodriguez-Hertz, M.A. Rodriguez-Hertz and R. Ures. Accessibility and stable ergodicity for partially hyperbolic diffeomorphism with 1D-center bundle. \textit{Invent. Math}, \textbf{172} (2008), 353-381.



\bi{Rug0} R. O. Ruggiero. On the creation of conjugate points, \textit{Mathematische Zeitschrift}, (1991), \textbf{208}, 41-55.

\bi{Rug2} R. O. Ruggiero. On a conjecture about expansive geodesic flows, \textit{Ergod. Theor. $\&$ Dynam. Sys.}, (1996), \textbf{16} (3), 545-553.

\bi{Rug1} R. O. Ruggiero. Expansive geodesic flows in manifolds with no conjugate points, \textit{Ergod. Theor. $\&$ Dynam. Sys.}, (1997), \textbf{17} (1), 211-225.	

\bi{Wi} A. Wilkinson. Stable ergodicity of the time-one map of a geodesic flow. Ergodic Theory and Dynamical Systems \textbf{18}, 1545-1587 (1998).

\end{thebibliography}
\end{document}